\documentclass[a4paper,reqno]{amsart}
\usepackage{amssymb, amsmath, amscd}
\def\deg{\operatorname{deg}}\def\ord{\operatorname{ord}}\def\Tr{\operatorname{Tr}}

\def\AA{{\mathcal{A}}}\def\QQ{{\mathcal{Q}}}\def\BB{{\mathcal{B}}}
\def\DeR{{\operatorname{DeR}}}
\def\gronk{\vphantom{\vrule height 12pt}}
\def\RR{{\mathcal{R}}}\def\TT{{\mathcal{T}}}
\def\PP{{\mathcal{P}}}\def\PPP{{\mathcal{P}_{m,k}}}
\def\UU{{\mathcal{U}}}\def\MMm{\mathcal{M}^m}\def\MMK{\mathcal{M}^{k+1}}\def\MM{\mathcal{M}}
\def\NN{\mathcal{N}}\def\xSS{{\mathcal{S}}}\def\NNk{\mathcal{N}^k}\def\NNv{\NN^\nu}
\def\PPF{{\mathfrak{P}}}\def\QQF{{\mathfrak{Q}}}
\def\PPFmU{{\mathfrak{P}_{m}^{U}}}\def\QQFmU{{\mathfrak{Q}_m^U}}
\def\PPFmkU{{\mathfrak{P}_{m,k}^{U}}}\def\QQFmkU{{\mathfrak{Q}_{m,k}^{U}}}
\def\PPFkkU{{\mathfrak{P}_{k,k}^{U}}}
\def\QQQ{\mathcal{Q}_{m, k}}
\def\xSSmk{{\mathcal{S}_{ m, k}}}
\def\xSSkk{{\mathcal{S}_{k,k}}}\def\xSSKk{{\mathcal{S}_{k+1,k}}}
\def\KKF{\mathfrak{K}}
\def\KKKFPmkU{{\mathfrak{K}_{\mathfrak{P},m, k}}}
\def\KKKFQmkU{{\mathfrak{K}_{\mathfrak{Q}, m, k}}}
\def\KKKFQKkU{{\mathfrak{K}_{\mathfrak{Q},k+1, k}}}

\def\QQQFU{{\mathfrak{Q}_{ m, k}^{U}}}
\def\xSSSFmk{{\mathfrak{S}_{ m, k}}}\def\xSSSFkk{{\mathfrak{S}_{k,k}}}
\def\xSSSFKk{{\mathfrak{S}_{k+1,k}}}

\def\xSSFm{{\mathfrak{S}_m}}\def\xSSFn{{\mathfrak{S}_n}}\def\xSSF{{\mathfrak{S}}}

\def\XiQKkT{\Xi_{\mathfrak{Q}, k+1, k}}
\def\XiQmkT{\Xi_{\mathfrak{Q}, m, k}}
\def\XiQmaT{\Xi_{\mathfrak{Q}, m, 1}}
\def\XiQmbT{\Xi_{\mathfrak{Q}, m, 2}}

\def\TQKk{\Theta_{\mathfrak{Q},k+1, k}}
\def\TQmk{\Theta_{\mathfrak{Q}, m, k}}
\def\XiPmk{{\Xi_{\mathfrak{P},m,k}}}
\def\XiPkk{{\Xi_{\mathfrak{P},k,k}}}

\newtheorem{theorem}{Theorem}[section]
\newtheorem{definition}{Definition}[section]
\newtheorem{lemma}{Lemma}[section]
\newtheorem{remark}{Remark}[section]

\newtheorem{example}{Example}[section]
\makeatletter
 \@addtoreset{equation}{section}
\makeatother
\begin{document}
\title[Universal curvature identities]
{Universal curvature identities and Euler Lagrange Formulas for K\"ahler manifolds}
\author{P. Gilkey, J.H. Park, and K. Sekigawa}
\address{PG: Mathematics Department, University of Oregon, Eugene OR 97403 USA}
\email{gilkey@uoregon.edu}
\address{JHP: Department of Mathematics, Sungkyunkwan University, Suwon
440-746, Korea \\\phantom{JHP:..A} Korea Institute for Advanced Study, Seoul
130-722, Korea} \email{parkj@skku.edu}
\address{KS: Department of Mathematics, Niigata University, Niigata, Japan.}
\email{sekigawa@math.sc.niigata-u.ac.jp}
\begin{abstract}{We relate certain universal curvature identities
for K\"ahler manifolds to the Euler-Lagrange
equations of the scalar invariants which are defined by
pairing characteristic forms with powers of the K\"ahler form.
}\end{abstract}
\maketitle

\section{Introduction}
Throughout this paper, we shall assume that $(M,g)$ is a
compact smooth oriented Riemannian manifold of dimension $2m$. Let $d\nu_g$
be the Riemannian volume $m$-form. In the introduction, we will establish the notation
that will enable us to state the two main results of this paper -- Theorem~\ref{T1.2}
(which describes the symmetric 2-tensor valued universal curvature identities in the
K\"ahler setting) and Theorem~\ref{T1.3} (which gives the Euler Lagrange equations for
the scalar invariants defined by pairing characteristic forms with powers of the K\"ahler form
in the K\"ahler setting). These two Theorems extend previous results from the
real setting to the K\"ahler setting as we shall discuss subsequently in Remark~\ref{R1.2}.

\subsection{K\"ahler geometry}
A {\it holomorphic structure} on $M$ is an endomorphism $J$ of the
tangent bundle $TM$ so that $J^2=-\operatorname{id}$ and
so that there exist local {\it holomorphic coordinate charts}
$(x^1,\dots,x^{ m},y^1,\dots,y^{ m})$ covering $M$ satisfying
$$J\partial_{x_\alpha}=\partial_{y_\alpha}\quad\text{and}\quad
   J\partial_{y_\alpha}=-\partial_{x_\alpha}\quad\text{for}\quad1\le\alpha\le m\,.
$$
Equivalently, via the Newlander-Nirenberg Theorem \cite{NN57},
this means that the Nijenhuis tensor $N_J$ vanishes
where one defines (see \cite{FN56}):
$$
N_{J}(X,Y):=[X,Y]+J[J X,Y]+J [X,J Y]-[J X,J Y]\,.
$$
In a system of holomorphic coordinates, we define for $1\le\alpha\le m$:
$$\begin{array}{ll}
z^\alpha:=x^\alpha+\sqrt{-1}y^\alpha,\\
\partial_{z_\alpha}:=\textstyle\frac12(\partial_{x_\alpha}-\sqrt{-1}\partial_{y_\alpha}),&
\partial_{\bar z_\alpha}:=\textstyle\frac12(\partial_{x_\alpha}+\sqrt{-1}\partial_{y_\alpha}),
\vphantom{\vrule height 12pt}\\
dz^\alpha:=dx^\alpha+\sqrt{-1}dy^\alpha,&
d\bar z^\alpha:=dx^\alpha-\sqrt{-1}dy^\alpha\,.\vphantom{\vrule height 12pt}
\end{array}$$
Extend $J$ to be complex linear on the complexified tangent bundle to obtain:
$$J\partial_{z_\alpha}=\sqrt{-1}\partial_{z_\alpha}\text{ and }
J\partial_{\bar z_\alpha}=-\sqrt{-1}\partial_{\bar z_\alpha}\,.$$

We can decompose the bundles $S^2M$ and $\Lambda^2M$ of symmetric and anti-symmetric
bilinear forms as $S^2M=S_+^2M\oplus S_-^2M$ and $\Lambda^2M=\Lambda_+^2M\oplus\Lambda_-^2M$ where
$$S_\pm^2M:=\{h\in S^2M:J^*h=\pm h\}\text{ and }
\Lambda_\pm^2M:=\{h\in \Lambda^2M:J^*h=\pm h\}\,.$$
A symmetric bilinear form $h\in S_+^2M$ is said to be {\it Hermitian}; if $h$
is Hermitian, then associated {\it K\"ahler form}
$\Omega_h\in \Lambda_+^2M$ is given by setting:
$$\Omega_h(x,y):=h(x,Jy)\,.$$
Conversely, given $\Omega\in\Lambda_+^2M$, we can recover $h=h_\Omega$
by setting $h(x,y)=\Omega(x,-Jy)$.
This correspondence defines a natural isomorphism between $S_+^2M$ and $\Lambda_+^2M$.

A triple $\MMm:=(M,g,J)$ is said to be a {\it Hermitian}
manifold if $g\in C^\infty(S_+^2M)$ is positive definite (and thus defines a
Riemannian metric on $M$) and if $(M,J)$ is a holomorphic of complex dimension $m$.
Let $\Omega=\Omega_g$. We then have that
\begin{equation}\label{E1.a}
d\nu_g=\textstyle\frac1{ m!}\Omega^{ m}\,.
\end{equation}
A Hermitian manifold $\MMm$ is said to be a {\it K\"ahler manifold}
if $d\Omega=0$. Let $\nabla$
be the Levi-Civiti connection and let
$$
\RR(x,y):=\nabla_x\nabla_y-\nabla_y\nabla_x-\nabla_{[x,y]}\text{ and }
R(x,y,z,w):=g(\RR(x,y)z,w)
$$
be the
{\it curvature operator} and the {\it curvature tensor}, respectively.
We shall also denote these tensors by $\RR_\MM$ and $R_\MM$ when
it is necessary to emphasize the role that $\MMm$ plays.
If $\MMm$ is a K\"ahler manifold, then $\nabla J=0$ and
we have an additional curvature symmetry called the {\it K\"ahler identity}:
\begin{equation}\label{E1.b}
\RR(x,y)J=J\RR(x,y)\text{ i.e. }R(x,y,z,w)=R(x,y,Jz,Jw)\,.
\end{equation}

\subsection{The characteristic classes and characteristic numbers}\label{sect-1.2}
Let $M_{ m}(\mathbb{C})$ be the matrix algebra of all $ m\times m$
complex matrices and let
$\operatorname{GL}_{ m}(\mathbb{C})\subset M_{ m}(\mathbb{C})$
be the associated general linear group.
Let $\xSSFm$ be the  ring of polynomial maps from
$M_{ m}(\mathbb{C})$ to $\mathbb{C}$ which
are invariant under the action of
$\operatorname{GL}_{ m}(\mathbb{C})$, i.e. $\xSS\in\xSSFm$
if and only if
$$
\xSS(ABA^{-1})=\xSS(B)\text{ for all }A\in\operatorname{GL}_{ m}(\mathbb{C})
\text{ and  for all }B\in M_{ m}(\mathbb{C})\,.$$
Define $\Tr_\mu\in\xSSFm$ by setting
$\Tr_\mu(B):=\Tr(B^\mu)$. We then have:
\begin{equation}\label{E1.c}
\xSSFm=\mathbb{C}[\Tr_1,\dots,\Tr_{ m}]\,.
\end{equation}
Let $\xSSSFmk\subset\xSSFm$
be the finite dimensional subspace of maps which
are homogeneous of degree $ k$.  We may then decompose
$$\xSSFm=\oplus_{ k}\xSSSFmk\,.$$
\begin{definition}\label{D1.1}
\rm Let $k$ be a positive integer.
A{\it partition} $\pi$ of $ k$ is a decomposition of $ k=n_1+\dots+n_\ell$
as the sum of positive integers where we order $n_1\ge\dots\ge n_\ell\ge 1$.
Let $\rho( k)$ be the {\it partition function}; this is the number of distinct partitions $\pi$
of $ k$. We use Equation~(\ref{E1.c}) to see that a basis for $\xSSSFmk$ consists of all
monomials of the form $\Tr_1^{\nu_1}\dots\Tr_m^{\nu_m}$ where $\nu_1+2\nu_2+\dots+m\nu_m=k$. Consequently
\begin{equation}\label{E1.d}
\dim\{\xSSSFmk\}=\rho( k)\text{ if } k\le m\,.
\end{equation}
Let $n<m$ and let $B_n\in M_n(\mathbb{C})$.
Let $0_\ell$ be the additive unit of $M_\ell(\mathbb{C})$. The natural map
$B_n\mapsto B_n\oplus 0_{ m-n}$ defines an inclusion of $M_n(\mathbb{C})$ into $M_m(\mathbb{C})$
and induces dually a {\it restriction map}
$r_{m,n}:\xSSFm\rightarrow\xSSFn$ which is characterized by the identity:
\begin{equation}\label{E1.e}
\{r_{m,n}(\xSS_m)\}(B_n):=\xSS_m(B_n\oplus 0_{m-n})\,.
\end{equation}
\end{definition}

\begin{remark}\label{R1.1}
\rm Let $n<m$. Since the restriction map preserves the grading,
$r_{m,n}$ maps $\xSSSFmk$ to $\mathfrak{S}_{n,k}$.
Since $\Tr\{B_n^i\}=\Tr\{(B_n\oplus0_{m-n})^i\}$, $r_{m,n}(\Tr_i)=\Tr_i$. Thus
 Equation~(\ref{E1.c}) shows that $r_{m,n}$ is always a surjective map from
$\mathfrak{S}_{m,k}$ to $\mathfrak{S}_{n,k}$. Furthermore, if $n\ge k$, then
$r_{m,n}$ is an isomorphism
from $\xSSSFmk$ to $\mathfrak{S}_{n,k}$.\end{remark}

Let $\MMm=(M,g,J)$ be a K\"ahler manifold. We use $J$ to give $TM$ a complex
structure and to regard $TM$ as a complex vector bundle; Equation~(\ref{E1.b}) then shows
that $\RR(x,y)$ is complex linear. We regard $\RR$ as a matrix of $2$-forms.
If $\xSSmk\in\xSSSFmk$, then the {\it evaluation} on $\RR$ yields an element
$$\xSSmk(\RR)\in C^\infty(\Lambda^{2 k}M)\,.$$
We have that $\xSSmk(\RR)$ is a closed differential form; the corresponding element
in de Rham cohomology is independent of the particular K\"ahler metric
$g$ on $M$ and is called a {\it characteristic class}:
$$[\xSSmk(\RR)]\in H^{2k}_{\DeR}(M)\,.$$

If $ k= m$, then
we may use the natural orientation of $M$ and integrate over $M$
to define a corresponding {\it characteristic number} which is independent of $g$.
If the complex dimension $ m=1$, then $\dim\{\xSSF_{1,1}\}=1$. If ${{\xSS}}_{1,1}\in\xSSF_{1,1}$,
then there is a universal constant $c=c(\xSS_{1,1})$ so that
$$\int_M\xSS_{1,1}(\RR_\MM) =c\cdot\chi(M)$$
where $\chi(M)$ is the {\it Euler-Poincar\'e} characteristic of $M$. Let
$\operatorname{sign}$ denote the Hirzebruch signature.
If the complex dimension $ m=2$, then $\dim\{\xSSF_{2,2}\}=2$. If
${{\xSS}}_{2,2}\in\xSSF_{2,2}$, then there are universal constants $c_i=c_i({{\xSS}}_{2,2})$ so that:
$$\int_M{{\xSS}}_{2,2}(\RR_\MM) =c_1\cdot\chi(M)+c_2\cdot\operatorname{sign}(M)\,.$$

Give complex projective space $\mathbb{CP}^n$
the Fubini-Study metric. If $\vec\nu=(\nu_1,\dots,\nu_\ell)$, let
$\mathbb{CP}^{\vec\nu}=\mathbb{CP}^{\nu_1}\times\dots\times\mathbb{CP}^{\nu_\ell}$.
This is a compact
homogeneous K\"ahler manifold of complex dimension $\nu_1+\dots+\nu_\ell$.
If $\xSSkk$ is non-trivial as an invariant polynomial,
then the associated characteristic number is non-trivial. We refer to \cite{APS73,H66} for
the proof of:
\begin{lemma}\label{L1.1}
\rm Let $0\ne\xSSkk\in\xSSSFkk$. Then there exists $\vec\nu$ with $k=\nu_1+\dots\nu_\ell$ so
$$\displaystyle\int_{\mathbb{CP}^{\vec\nu}}\xSS_{ k, k}
(\RR_{\mathbb{CP}^{\vec\nu}})d\nu_{\mathbb{CP}^{\vec\nu}} \ne0\,.$$
\end{lemma}

\subsection{Scalar valued universal curvature identities}

In the real setting,Weyl's first theorem of invariants \cite{W46}
can be used to show that all polynomial
scalar invariants in the derivatives of the metric arise from contractions
of indices in the curvature tensor and its
covariant derivatives. Let $\{e_i\}$ be a local orthonormal frame for a
Riemannian manifold $(M,g)$ and let $R_{ijkl}$
be the components of the curvature tensor. Adopt the {\it Einstein convention}
and sum over repeated indices to define:
\medbreak\qquad
$E_2:=R_{ijji}$,\qquad
$E_4:=R_{ijji}R_{kllk}-4R_{aija}R_{bijb}+R_{ijkl}R_{ijkl}$, and
\medbreak\qquad
$E_6:=R_{ijji}R_{kllk}R_{abba}-12R_{ijji}R_{aija}R_{bijb}+3R_{abba}R_{ijkl}R_{ijkl}$
\smallbreak
$\qquad\qquad+24R_{aija}R_{bklb}R_{jlik}
+16R_{aija}R_{bjkb}R_{cikc}-24R_{aija}R_{jkln}R_{lnik}$
\smallbreak
$\qquad\qquad+2R_{ijkl}R_{klan}R_{anij}-8R_{kaij}R_{inkl}R_{jlan}$.
\medbreak\noindent
 $E_2$, $E_4$, and $E_6$
 are universally defined scalar invariants of order $\mu=2$, $\mu=4$, and $\mu=6$,
respectively. They
are generically non-zero in real dimension at least $\mu$ but vanish
in lower dimensions; in particular,
they give non-trivial {\it universal curvature identities} in real dimension $\mu-1$.
Modulo a suitable normalization, these are the integrals
of the Chern-Gauss-Bonnet Theorem \cite{C44}
and more generally,
up to rescaling,
the Pfaffian $E_\mu$ gives the only universal curvature identity
of order $\mu$ vanishing identically
in real dimension $\mu-1$.
This fact plays an important role in the proof of the Chern-Gauss-Bonnet theorem
using heat equation methods \cite{G73}.

\begin{definition}\label{D1.2}
\rm Let $\PPF_{ m}$ be the polynomial algebra in the components
of $\RR$, in the components of the covariant derivative $\nabla\RR$, and so forth
 for K\"ahler metrics on manifolds of complex dimension $m$. Let $\PPFmkU$ be the subspace of
polynomials which are homogeneous of degree $2k$ in the derivatives
of the metric and which are invariant under the action of the
unitary group $U( m)$.
\end{definition}
H. Weyl's theorem on invariants of the orthogonal group \cite{W46}
has been extended by Fukami \cite{Fu58} and Iwahori \cite{Iw58} to this setting; all such invariants
arise by contractions of indices using the metric and the K\"ahler form.
In practice, the K\"ahler identity  means that we will not be in fact using the K\"ahler form
to contract indices. Rather, we will contract a lower holomorphic (resp. anti-holomorphic) index against
the corresponding upper holomorphic (resp. anti-holomorphic) index. Thus,
the K\"ahler form $\Omega=-\sqrt{-1}g_{\alpha\bar\beta}dz^{\alpha}\wedge d\bar z^{\beta}$
is given
 by contracting upper against lower indices of the same type; it is not necessary for the
frame to be unitary. We can also contract a lower holomorphic index against a corresponding
lower anti-holomorphic index using the metric relative to a unitary frame.
Thus, for example, the scalar curvature is given
by $\tau=R_{\alpha\bar\alpha\bar\beta\beta}$ modulo a suitable normalizing constant.

\begin{definition}\label{D1.3}
\rm Let $\PPFmkU$ be as defined in Definition~\ref{D1.2}.
Let $\KKKFPmkU\subset\PPFmkU$ be the subspace of invariant local formulas
which are homogeneous of degree $2k$ in the derivatives of the metric and which
vanish when restricted from complex dimension $m$ to complex dimension $k-1$; we shall
give an algebraic characterization presently in Lemma~\ref{L3.1}.\end{definition}

Elements $0\ne\PPP\in\KKKFPmkU$ give universal
curvature identities of degree $2k$ in complex dimension $k-1$.
We sum over repeated
indices in a unitary frame field to define:
\begin{eqnarray*}
&&\PP_{ m,2}^1:=R_{\alpha_1\bar\alpha_1\bar\alpha_3\alpha_4}
R_{\alpha_2\bar\alpha_2\bar\alpha_4\alpha_3}
-R_{\alpha_1\bar\alpha_2\bar\alpha_3\alpha_4}R_{\alpha_2\bar\alpha_1\bar\alpha_4\alpha_3}\,,\\
&&\PP_{ m,2}^2:=R_{\alpha_1\bar\alpha_1\bar\alpha_3\alpha_3}
R_{\alpha_2\bar\alpha_2\bar\alpha_4\alpha_4}
-R_{\alpha_1\bar\alpha_2\bar\alpha_3\alpha_3}R_{\alpha_2\bar\alpha_1\bar\alpha_4\alpha_4}\,.
\end{eqnarray*}
One then has that $P_{ m,2}^1$ and $P_{ m,2}^2$ are
generically non-zero if $ m\ge2$ but vanish
identically in complex dimension $ m=1$. Thus $\PP_{m,2}^1$ and $\PP_{m,2}^2$
 are universal curvature identities in the K\"ahler setting.
One sees this
not by using index notation but by noting that:
$$
\PP_{ m,2}^1:=\textstyle\frac12g(\Tr\{\RR^2\},\Omega^2)\text{ and }
\PP_{ m,2}^2:=\textstyle\frac12g(\Tr\{\RR\}^2,\Omega^2)\,.
$$
We generalize this construction:
\begin{definition}\label{D1.4}
\rm If $\xSSmk\in\xSSSFmk$, define $\XiPmk:\xSSSFmk\rightarrow\PPFmkU$ by setting:
\begin{equation}\label{E1.f}
\XiPmk(\xSSmk):=\textstyle\frac1{ k!}g(\xSSmk(\RR),\Omega^{ k})\,.
\end{equation}
We may use Equation~(\ref{E1.a}) to see that if $m=k$, then
\begin{equation}\label{E1.g}
\Xi_{\mathfrak{P},m,m}(\mathcal{S}_{m,m})d\nu_g=\mathcal{S}_{m,m}(\mathcal{R})\,.
\end{equation}
Thus by Lemma~\ref{L1.1}, $\XiPmk(\xSSmk)$ is generically non-zero
in complex dimension $ m\ge k$ but vanishes
in complex dimension $ m= k-1$. Consequently, $\XiPmk$ takes values in $\KKKFPmkU$.
\end{definition}

 The following result played an important role in the proof
of the Riemann-Roch formula using heat equation methods \cite{G73a}:
\begin{theorem}\label{T1.1}
If $m\ge k$, then
$\XiPmk$ is an isomorphism from $\xSSSFmk$ to $\KKKFPmkU$.
In other words, any scalar valued curvature identity of order $2 k$ that is given
universally by contracting indices in pairs,
that is generically non-zero in complex dimension
$m\ge k$, and that vanishes in complex dimension $m=k-1$ is of this form.
\end{theorem}

\subsection{Universal curvature identities which are symmetric 2-tensor valued}
In the real setting, let $S^2M\subset\otimes^2T^*M$ be the bundle of symmetric $2$-cotensors
and let $S_2M\subset\otimes^2TM$ be the dual bundle; this is the bundle of symmetric $2$-tensors.
We can extend H. Weyl's theorem first theorem of invariants to
construct polynomial invariants which are $S_2M$
valued by contracting all but 2 indices
and symmetrizing the remaining two indices. For example, we can define:
\begin{eqnarray*}
&&T_2:=R_{ijji}e_k\circ e_k-2R_{ijki}e_j\circ e_k,\\
&&T_4:=-\textstyle\frac14(R_{ijji}R_{kllk}-4R_{ijki}R_{ljkl}+R_{ijkl}R_{ijkl})e_n\circ e_n\\
&&\qquad\qquad+\{R_{klni}R_{klnj}-2R_{knik}R_{lnjl}
-2R_{iklj}R_{nkln}+R_{kllk}R_{nijn}\}e_i\circ e_j\,.
\end{eqnarray*}
The invariants $T_n$ are generically non-zero in real dimension greater than $n$ but
vanish identically in real dimension $n$. The identity $T_2=0$ in real dimension $2$
is the classical identity relating the scalar
curvature and the Ricci tensor; the identity $T_4=0$ in real dimension $4$
is the Berger-Euh-Park-Sekigawa identity \cite{M70,EPS10}.
More generally, such invariants can be formed through the transgression of the
Euler form; we refer to \cite{GPS11} for further details. We also refer to
\cite{GPS12} where the pseudo-Riemannian setting is treated and to \cite{GPS13}
where manifolds with boundary are treated. We note that Navarro and Navarro
\cite{NN13} have applied the theory of  natural operators \cite{KPS93,N13}
to discuss more generally
$p$-covariant identities for any even $p$.

In the K\"ahler setting, let $S^+_2M$ be the bundle dual to $S_+^2M$
and let $\langle\cdot,\cdot\rangle$ denote the natural pairing between these two bundles.
\begin{definition}\label{D1.5}
\rm Let $\QQQFU$ be the space of all $ S_+^2$ valued invariants
which are homogeneous of degree
$2 k$
in the derivatives of the metric and which
are invariant under the action of the unitary group.
We consider the subspace $\KKKFQmkU\subset\QQQFU$ of invariants which
vanish when restricted from complex dimension $m$ to complex dimension $k$;
again, we shall give an algebraic characterization presently in Lemma~\ref{L3.1}.\end{definition}

\begin{example}\label{E1.1}
\rm Let $\{e_\alpha\}$ be a local unitary frame field for $TM$
(viewed as a complex
vector bundle).
We contract holomorphic with anti-holomorphic indices in pairs to construct
the following invariant of degree 2:
$$\mathcal{Q}_{m,1}:=R_{\alpha \bar \alpha_1 \bar r_1 r_1}e_{\alpha_2}\circ
e_{\bar\alpha_2} -R_{\alpha_1 \bar \alpha_2 \bar r_1
r_1}e_{\alpha_2}\circ e_{\bar\alpha_1}\,.
$$
Similarly, we may construct invariants of degree 4: \medbreak\qquad
$\mathcal{Q}_{m,2}^1:=R_{\alpha_1\bar\alpha_1\bar\gamma_1\delta_1}
R_{\alpha_2\bar\alpha_2\bar\delta_1\gamma_1} e_{\alpha_3}\circ
e_{\bar\alpha_3}
+R_{\alpha_1\bar\alpha_3\bar\gamma_1\delta_1}
R_{\alpha_2\bar\alpha_1\bar\delta_1\gamma_1}
e_{\alpha_3}\circ e_{\bar\alpha_2}$ \smallbreak\qquad\qquad
%%%
{{$
+R_{\alpha_1\bar\alpha_2\bar\gamma_1\delta_1}R_{\alpha_2\bar\alpha_3\bar\delta_1\gamma_1}
e_{\alpha_3}\circ e_{\bar\alpha_1}
-R_{\alpha_1\bar\alpha_1\bar\gamma_1\delta_1}R_{\alpha_2\bar\alpha_3\bar\delta_1\gamma_1}
e_{\alpha_3}\circ e_{\bar\alpha_2}$}}
%%%
%$+R_{\alpha_1\bar\beta_1\bar\gamma_1\delta_1}R_{\alpha_2\bar\beta_2\bar\delta_1\gamma_1}
%e_{\alpha_3}\circ e_{\bar\beta_3}
%-R_{\alpha_1\bar\alpha_1\bar\gamma_1\delta_1}R_{\alpha_2\bar\alpha_2\bar\delta_1\gamma_1}
%e_{\alpha_3}\circ e_{\bar\alpha_3}$
\smallbreak\qquad\qquad
$-R_{\alpha_1\bar\alpha_2\bar\gamma_1\delta_1}R_{\alpha_2\bar\alpha_1\bar\delta_1\gamma_1}
e_{\alpha_3}\circ e_{\bar\alpha_3}
-R_{\alpha_1\bar\alpha_3\bar\gamma_1\delta_1}R_{\alpha_2\bar\alpha_2\bar\delta_1\gamma_1}
e_{\alpha_3}\circ e_{\bar\alpha_1}$, \medbreak\qquad
$\mathcal{Q}_{m,2}^2:=R_{\alpha_1\bar\alpha_1\bar\sigma_1\sigma_1}
R_{\alpha_2\bar\alpha_2\bar\sigma_2\sigma_2} e_{\alpha_3}\circ
e_{\bar\alpha_3}
+R_{\alpha_1\bar\alpha_3\bar\sigma_1\sigma_1}R_{\alpha_2\bar\alpha_1\bar\sigma_2\sigma_2}
e_{\alpha_3}\circ e_{\bar\alpha_2}$ \smallbreak\qquad\qquad
 {{$
+R_{\alpha_1\bar\alpha_2\bar\sigma_1\sigma_1}R_{\alpha_2\bar\alpha_3\bar\sigma_2\sigma_2}
e_{\alpha_3}\circ e_{\bar\alpha_1}
-R_{\alpha_1\bar\alpha_1\bar\sigma_1\sigma_1}R_{\alpha_2\bar\alpha_3\bar\sigma_2\sigma_2}
e_{\alpha_3}\circ e_{\bar\alpha_2}$}}
%$+R_{\alpha_1\bar\beta_1\bar\sigma_1\sigma_1}R_{\alpha_2\bar\beta_2\bar\sigma_2\sigma_2}
%e_{\alpha_3}\circ e_{\bar\beta_3}
%-R_{\alpha_1\bar\alpha_1\bar\sigma_1\sigma_1}R_{\alpha_2\bar\alpha_2\bar\sigma_2\sigma_2}
%e_{\alpha_3}\circ e_{\bar\alpha_3}$
\smallbreak\qquad\qquad
$-R_{\alpha_1\bar\alpha_2\bar\sigma_1\sigma_1}R_{\alpha_2\bar\alpha_1\bar\sigma_2\sigma_2}
e_{\alpha_3}\circ e_{\bar\alpha_3}
-R_{\alpha_1\bar\alpha_3\bar\sigma_1\sigma_1}R_{\alpha_2\bar\alpha_2\bar\sigma_2\sigma_2}
e_{\alpha_3}\circ e_{\bar\alpha_1}$.
\medbreak\noindent
We have $\mathcal{Q}_{m,1}\in\KKF_{\mathfrak{Q},m,1}$,
$\mathcal{Q}_{m,2}^1\in\KKF_{\mathfrak{Q},m,2}$ and
$\mathcal{Q}_{m,2}^2\in\KKF_{\mathfrak{Q},m,2}$.
The invariant $\mathcal{Q}_{m,1}$ is generically non-zero in complex dimension
 $m\ge2$ but vanishes in complex dimension $m=1$;
the invariants $\mathcal{Q}_{m,2}^1$ and $\mathcal{Q}_{m,2}^2$ are generically non-zero in
complex dimension
$ m\ge3$ but vanish in complex dimension $ m=2$.
One sees this not by using the index notation but rather by expressing
\begin{eqnarray*}
&&\mathcal{Q}_{m,1}={\textstyle\frac12
R_{\alpha_1\bar\beta_1\bar\gamma_1\gamma_1}e_{\alpha_2}\circ
e_{\bar\beta_2}} g(dz^{\alpha_1}\wedge d \bar
z^{\beta_1}\wedge dz^{\alpha_2}\wedge d \bar z^{\beta_2},
\Omega^2),\\
&&\mathcal{Q}_{m,2}^1=\textstyle\frac16
R_{\alpha_1\bar\beta_1\bar\gamma_1\delta_1}R_{\alpha_2\bar\beta_2\bar\delta_1\gamma_1}
e_{\alpha_3}\circ e_{\bar\beta_3}g(e^{\alpha_1}\wedge\bar e^{\beta_1}\wedge e^{\alpha_2}
\wedge\bar  e^{\beta_2}\wedge e^{\alpha_3}\wedge \bar e^{\beta_3},\Omega^3),\\
&&\mathcal{Q}_{m,2}^2=\textstyle\frac16
R_{\alpha_1\bar\beta_1\bar\sigma_1\sigma_1}R_{\alpha_2\bar\beta_2\bar\sigma_2\sigma_2}
e_{\alpha_3}\circ e_{\bar\beta_3}g(e^{\alpha_1}\wedge\bar e^{\beta_1}\wedge e^{\alpha_2}
\wedge\bar  e^{\beta_2}\wedge e^{\alpha_3}\wedge \bar e^{\beta_3},\Omega^3)\,.
\end{eqnarray*}\end{example}

We generalize this construction:

\begin{definition}\label{D1.6}
\rm
Let $\xSSmk\in\xSSSFmk$. The {\it transgression} $\XiQmkT(\xSSmk)\in S_2^+$
is defined by setting:
$$
\XiQmkT(\xSSmk):={\textstyle\frac1{(k+1)!}}g(\xSSmk(\RR)\wedge e^\alpha\wedge\bar e^\beta,
\Omega^{k+1})e_\alpha\circ\bar e_\beta\,.
$$
\end{definition}

\begin{example}\label{E1.2}
\rm Adopt the notation of Example~\ref{E1.1}. Then
$\mathcal{Q}_{m,1}=\XiQmaT(\Tr_1)$.
Let $\rho$ be the Ricci tensor and let $\tau$ be the scalar curvature. We have
$$\mathcal{Q}_{m,1}=-\textstyle\frac12\tau g+\rho\,.$$
This symmetric 2-form valued tensor is generically non-zero if $m\ge2$ but vanishes identically
in complex dimension $m=1$; this is a classic identity. Recall that
$\Tr_k(R)=\Tr(R^k)$. Let
$\mathcal{Q}_{m,2}^1=\XiQmbT(\Tr_2)$ and
$\mathcal{Q}_{m,2}^2=\XiQmbT(\Tr_1^2)$. Let $\rho$ be the Ricci tensor. Set
\begin{eqnarray*}
\check{R}_{ij}=R_{abci}{R^{abc}}{}_{j},\quad
\check{\rho}_{ij}=\rho_{ai}\rho^{a}{}_j,\quad
     L_{ij}=2R_{iabj}\rho^{ab}\,.
\end{eqnarray*}
We then have:
$$\mathcal{Q}_{m,2}^1=\textstyle(\frac{1}{2}|\rho|^2-\frac{1}{4}|R|^2)g+(\check{R}-L(\rho))\text{ and }
\mathcal{Q}_{m,2}^2=2\check{\rho}-\tau\rho-\frac{1}{2}(|\rho|^2-\frac{\tau^2}{2})g\,.$$
The characteristic class $c_1^2$ corresponds to $\Tr_2$; the formula
for $\mathcal{Q}_{m,2}^2$ agrees with that given in Theorem~5.3 \cite{EPS13} for
the associated Euler-Lagrange equation.
Furthermore, the Euler class in real dimension $4$ corresponds to
$2\det(A)=\mathcal{Q}_{m,2}^2-\mathcal{Q}_{m,2}^1$. We express:
$$\mathcal{Q}_{m,2}^2-\mathcal{Q}_{m,2}^1= \textstyle\frac{1}{4}(|R|^2-|\rho|^2+\frac{\tau^2}{4})g-\check{R}+L(\rho)+2\check{\rho}-\tau\rho\,.
$$
This is the universal curvature identity discussed in \cite{M70,EPS10} that
is generated by the Euler-Lagrange equation of this characteristic class.
Note that the complex structure is not involved; this is no longer the case
when we consider invariants of order 6 and higher.\end{example}

The invariants of Definition~\ref{D1.6} yield the universal $ S_2^+$ valued curvature
identities that we have been searching for; every $ S_2^+$ valued invariant which is
homogeneous of degree $2k$ in the derivatives of the metric and which is generically
non-zero in complex dimension $m>k$ and which vanishes in complex dimension $k$ arises
in this fasion. Theorem~\ref{T1.1} generalizes to
this setting to become the following result which is the first major
new result of this paper:

\begin{theorem}\label{T1.2}
If $m>k$, then map $\XiQmkT$ of Definition~\ref{D1.6}
 is an isomorphism from $\xSSSFmk$ to $\KKKFQmkU$.
This means that a $S _2^+$ valued curvature identity of order $2 k$
which is given universally by contracting indices in pairs, which is
generically non-zero in complex dimension $m>k$, and which vanishes
in complex dimension $m=k$ is of this form.
\end{theorem}

\subsection{Euler Lagrange equations}
Let $\MMm=(M,g,J)$ be a compact K\"ahler manifold.
Let $\xSSmk\in\xSSSFmk$ for $ k\le m$. Although $\xSSmk$
determines a cohomology class, it does not determine
a corresponding scalar invariant if $ k< m$.
We integrate the invariant of Definition~\ref{D1.4} to define:
\begin{equation}\label{E1.h}
\{\XiPmk({{\xSSmk}})\}[\MMm]:=\frac1{ k!}\int_Mg(\xSSmk(\RR_\MM),\Omega_g^{ k})d\nu_g\,.
\end{equation}
If $ k= m$, we use Equation~(\ref{E1.g}) to see
$$\{\XiPkk({{\xSSkk}})\}[\MMm]=\int_M{{\xSSkk}}(\mathcal{R}_\MM)$$
is a characteristic number that is independent of the metric $g$. However,
more generally, if $m>k$, then this integral depends upon the metric.
Let
$g_\varepsilon:=g+\varepsilon h$ be a smooth 1-parameter family of K\"ahler metrics;
such families may be obtained using the K\"ahler potential
as we shall discuss presently in Section~\ref{S2.3}.
We integrate by parts to
obtain the corresponding {\it Euler-Lagrange} formula.
\begin{definition}\label{D1.7}
\rm
Let $\xSSmk\in\xSSSFmk$. Let $\MMm=(M,g,J)$ be a K\"ahler manifold of complex dimension $m$. Let
$\MMm_\varepsilon:=(M,g+\varepsilon h,J)$ be a K\"ahler variation. Let
$\TQmk\{\xSSmk\}\in S^+_2M$
be the associated {\it Euler-Lagrange invariant}; it is uniquely characterized by the identity:
$$\partial_{\varepsilon}\left.\left\{\XiPmk({{\xSSmk}})[\MMm_\varepsilon]\right\}\right|_{\varepsilon=0}
=\int_M\left\langle\vphantom{\vrule height 11pt}
\left\{\TQmk(\xSSmk)\right\}(\RR_\MM),h\right\rangle d\nu_g\,.
$$
\end{definition}

What is perhaps somewhat surprising is that the Euler-Lagrange formulas
for $\xSSmk$ are closely related
to the universal curvature identities defined by the transgression.
We adopt the notation of Example~\ref{E1.1} and Example~\ref{E1.2}.
It is well known that $\mathcal{Q}_{m,1}$ is the Euler-Lagrange equation for
the Gauss-Bonnet integrand.
Work of \cite{EPS13}
shows that universal curvature identity
$\mathcal{Q}_{m,2}^1$ is the Euler-Lagrange equation for $\Tr_2$.
Similarly, work of \cite{M70,EPS10} shows that
 the universal curvature identity $\mathcal{Q}_{m,2}^2-\mathcal{Q}_{m,2}^1$
is the Euler-Lagrange equation of the Euler class. Thus $\XiQmkT=\TQmk$ if $k=1,2$.
This is true more generally; the map from the characteristic forms to the symmetric 2-tensors
given by the Euler-Lagrange equations coincides with the map given algebraically by the transgression
in the K\"ahler setting. Let $\TQmk$ be as given
in Definition~\ref{D1.7} and let $\XiQmkT$ be as given in Definition~\ref{D1.6}.
The following is the second main result of this paper:

\begin{theorem}\label{T1.3}
If $m>k$, then
$\TQmk=\XiQmkT$. This means that if $\xSSmk\in\xSSSFmk$, if $m>k$, and if
$\MMm_\varepsilon:=(M,g+\varepsilon h,J)$ is a K\"ahler variation, then
\begin{eqnarray*}
&&\partial_{\varepsilon}\left.\left\{\XiPmk({{\xSSmk}})[\MMm_\varepsilon]\right\}\right|_{\varepsilon=0}\\
&=&\frac1{(k+1)!}
\int_Mg(\xSSmk(\RR)\wedge e^{\alpha}\wedge\bar e^\beta,
\Omega^{k+1})\langle e_\alpha\circ\bar e_\beta,h\rangle
d\nu_\MM\,.
\end{eqnarray*}
\end{theorem}

\begin{remark}\label{R1.2}
\rm A-priori, since the local invariant $\xSSmk$ involves
$2^{\operatorname{nd}}$ derivatives, the associated Euler-Lagrange invariant
could involve the first and second covariant derivatives
of the curvature tensor. The somewhat surprising fact is that
this is not the case as Theorem~\ref{T1.3} shows.
In the real setting,
one can work with the {\it Pfaffian}; this is the integrand of the
Chern-Gauss-Bonnet formula \cite{C44}.
Berger \cite{M70} conjectured that the corresponding Euler-Lagrange invariant
only involved the second derivatives of the
metric. This was established
by Kuz'mina \cite{K74} and Labbi \cite{L05,L07,L08} (see also the discussion in \cite{GPS11}).
Theorem~\ref{T1.3} is the extension to the complex setting of this result.
\end{remark}

\subsection{Outline of the paper}
Fix a point of a K\"ahler manifold $\MMm$. In Section~\ref{sect-2}, we normalize the choice
of the coordinate system to be the unitary group up to arbitrarily high order. In Section~\ref{S3},
we give an algebraic description of the space $\KKKFPmkU$ (resp. $\KKKFQmkU$) from the
point of the restriction map from complex dimension $m$ to complex dimension $k-1$ (resp. $k$)
and show that $\XiPmk$ (resp. $\XiQmkT$ and $\TQmk$) takes values in $\KKKFPmkU$
(resp. $\KKKFQmkU$).
In Section~\ref{S4} we discuss invariance theory. We take a slightly non-standard point of view.
Weyl's first theorem of invariants \cite{W46} gives generators
for the space of invariants of the orthogonal group;
in brief, this generating set can be described in terms of contractions of indices.
Fukami \cite{Fu58} and Iwahori \cite{Iw58} have extended this result to the complex setting;
the generating set is formed by using both the metric and the K\"ahler form to contract indices.
However, what is needed in our analysis is Weyl's second theorem
of invariants which describes the relations among
the generating set described above. This analysis does not seem
to have been extended to the complex setting. Even were this to have been done,
we would still need to use the K\"ahler identity suitably. For that reason,
 it seemed easiest simply to do the necessary invariance theory
from scratch in a non-standard setting and we apologize in advance
if this is unfamiliar. Let $\KKKFQmkU$ be as given in Definition~\ref{D1.5}
and let $\rho(k)$ be the partition function of Definition~\ref{D1.1}.
The crucial estimate in this regard is given in Lemma~\ref{L4.3}:
$$\dim\{\KKKFQmkU\}\le\rho( k)\,.$$
In Section~\ref{S5},
we use these results of Section~\ref{S3} to establish Theorem~\ref{T1.1},
Theorem~\ref{T1.2}, and Theorem~\ref{T1.3}.

\section{Normalizing the coordinates}\label{sect-2}
In this section, we probe in a bit more detail into K\"ahler geometry. In Section~\ref{S2.1},
we introduce some basic notational conventions. In Section~\ref{S2.2}, we reduce the
structure group to the unitary group modulo a holomorphic transformation of arbitrarily high
order. In Section~\ref{S2.3}, we discuss K\"ahler potentials; this provides a way of varying
the original K\"ahler metric that will be very useful in considering the Euler-Lagrange equations.
In Section~\ref{S2.4}, we will use the K\"ahler potential to specify the jets of the metric; we shall
work with a polynomial algebra in the derivatives of the metric and in this section, we show
there are no hidden relations or analogues of the Bianchi identities. This will be crucial in our
subsequent discussion in Section~\ref{S3}.
\subsection{Notational conventions}\label{S2.1}
Let $P$ be a point of a K\"ahler manifold $\MMm$. Extend the
$J$-invariant Riemannian metric $g$ to be a symmetric
complex bilinear form. Let
$$g_{\alpha\beta}:=g(\partial_{z_\alpha},\partial_{z_\beta}),\quad
    g_{\bar\alpha\bar\beta}:=g(\partial_{\bar z_{\alpha}},\partial_{\bar z_{\beta}}),\quad
    g_{\alpha\bar\beta}:=g(\partial_{z_{\alpha}},\partial_{\bar z_{\beta}}).
$$
Since $g$ is $J$-invariant, we may show that $g_{\alpha\beta}=g_{\bar\alpha\bar\beta}=0$ by computing:
\begin{eqnarray*}
&&g_{\alpha\beta}=g(J\partial_{z_\alpha},J\partial_{z_\beta})=
g(\sqrt{-1}\partial_{z_\alpha},\sqrt{-1}\partial_{z_\beta})=-g_{\alpha\beta},\\
&&g_{\bar\alpha\bar\beta}=g(J\partial_{\bar z_{\alpha}},J\partial_{\bar z_{\beta}})=
g(-\sqrt{-1}\partial_{\bar z_{\alpha}},-\sqrt{-1}\partial_{\bar z_{\beta}})=-g_{\bar\alpha\bar\beta}
\end{eqnarray*}
As a result, we have that:
\begin{eqnarray*}
&&\Omega(\partial_{z_\alpha},\partial_{z_\beta})=g(\partial_{z_\alpha},J\partial_{z_\beta})
=\sqrt{-1}g_{\alpha\beta}=0,\\
&&\Omega(\partial_{\bar z_\alpha},\partial_{\bar z_\beta})
=g(\partial_{\bar z_\alpha},J\partial_{\bar z_\beta})
=-\sqrt{-1}g_{\bar\alpha\bar\beta}=0,\\
&&\Omega(\partial_{z_\alpha},\partial_{\bar z_\beta})=
g(\partial_{z_\alpha},J\partial_{\bar z_\beta})=
-\sqrt{-1}g_{\alpha\bar\beta},\\
&&\Omega=-{\sqrt{-1}}g_{\alpha\bar\beta}dz^\alpha\wedge d\bar z^\beta\,.
\end{eqnarray*}
The equation $d\Omega=0$ is then equivalent to the relations
\begin{equation}\label{eqn-2.a}
\begin{array}{l}
0=\partial_{z_\gamma}g_{\alpha\bar\beta}dz^\gamma\wedge dz^\alpha\wedge d\bar z^\beta
-\partial_{\bar z_{\gamma}}g_{\alpha\bar\beta}dz^\alpha\wedge
d\bar z^{\gamma}\wedge d\bar z^\beta,\quad\text{i.e.}\\
\partial_{z_\gamma}g_{\alpha\bar\beta}=\partial_{z_\alpha}g_{\gamma\bar\beta}\text{ and }
\partial_{\bar z_{\gamma}}g_{\alpha\bar\beta}=\partial_{\bar z_{\beta}}g_{\alpha\bar\gamma}\,.
\vphantom{\vrule height 11pt}\end{array}\end{equation}

Let $\delta$ be the Kronecker symbol.
Let $A:=(\alpha_1,\dots,\alpha_\nu)$ be
an ordered collection of indices $\alpha_i$ where $1\le\alpha_i\le m$. Let
$|A|=\nu$, let $z^A=z^{\alpha_1}\dots z^{\alpha_\nu}$, and let
$$\deg_\alpha(A):=\delta_{\alpha\alpha_1}+\dots+\delta_{\alpha\alpha_\nu}
$$
be the number of times the index $\alpha$ appears in $A$.
Let $B=(\beta_1,\dots,\beta_\mu)$ be another collection of indices and
let $\vec z=(z^1,\dots,z^{ m})$
be a local holomorphic system of coordinates on a K\"ahler manifold $\MMm$. Set
$$g^{\vec z}(A;B):=\left\{\partial_{z_{\alpha_2}}\dots
\partial_{z_{\alpha_\nu}}\partial_{\bar z_{\beta_2}}\dots
\partial_{\bar z_{\beta_\mu}}\right\} g_{\alpha_1\bar\beta_1}\,.$$
We shall often omit the superscript $\vec z$ if there is only one coordinate system under consideration.
If $\sigma$ and $\tau$ are permutations, let
$$
A^\sigma:=(\alpha_{\sigma(1)},\dots,\alpha_{\sigma(\nu)})\text{ and }
B^\tau:=(\beta_{\tau(1)},\dots,\beta_{\tau(\mu)})\,.
$$
Equation~(\ref{eqn-2.a}) may then be differentiated to see:
$$g(A;B)=g(A^\sigma;B^\tau)$$
so the variables $g(A;B)$ are symmetric in the holomorphic indices
and also in the anti-holomorphic indices;
the order of the indices comprising $A$ and comprising $B$ plays no role. Note that
$$g^{\vec z}(B;A)=\bar g^{\vec z}(A;B)\,.$$

\subsection{Reducing the structure group to $U( m)$}\label{S2.2}
The following result will enable us to normalize
the structure group of admissible coordinate transformations
from the full group of holomorphic transformations
to the unitary group
modulo changes which vanish to arbitrarily high order at a given point $P$ of $M$:

\begin{lemma}\label{lem-2.1}
 Let $P$ be a point of a K\"ahler manifold $\MMm$. Fix $n$.
\begin{enumerate}
\item There exist
local holomorphic coordinates $(z^1,\dots,z^{ m})$
 centered at $P$ so that
 \begin{equation}\label{eqn-2.b}
 \phantom{gronk}g_{\alpha\bar\beta}(P)=\delta_{\alpha\beta}\text{ and }
 g^{\vec z}(A;B)(P)=0\text{ for }|B|=1\text{ and }2\le |A|\le n\,.\end{equation}
\item If $(w^1,\dots,w^{ m})$ is another system of local holomorphic coordinates on $M$
which are centered at $P$ and which satisfy the relations of Equation~(\ref{eqn-2.b}), then
$z=Tw+O(|w|^{n+1})$ for some linear map $T\in U( m)$.
\end{enumerate}
\end{lemma}

\begin{proof} Suppose that $n=1$. We use the Gram-Schmidt process to make a
complex linear change of coordinates to ensure
that $g_{\alpha\bar\beta}(P)=\delta_{\alpha\beta}$. Assertion~(1) now follows;
Assertion~(2) is then immediate.
We therefore proceed by induction and assume that $n\ge2$.
Let $z$ be a system of coordinates normalized satisfying
$g_{\alpha\bar\beta}(P)=\delta_{\alpha\beta}$ and $g(A;B)=0$ for $|B|=1$ and
$2\le|A|<n$
(this condition is vacuous if $n=2$).
Consider the coordinate transformation:
$$w^\beta=z^\beta+\sum_{|A|=n}c_A^\beta z^A$$
where the constants $c_A^\beta$ are to be chosen suitably.  Set
\begin{equation}\label{eqn-2.c}
\varepsilon(A):=\partial_{z_{\alpha_1}}\dots\partial_{z_{\alpha_\nu}}\{z^A\}\in\mathbb{N}\,.
\end{equation}
We sum over repeated indices to compute:
\begin{eqnarray*}
&&\partial_{z_\alpha}=\partial_{w_\alpha}
+c_A^\gamma\partial_{z_\alpha}\{z^A\}\partial_{w_\gamma},\quad
\partial_{\bar z_{\beta}}=\partial_{\bar w_\beta}
+\bar c_A^\gamma\partial_{\bar z_\beta}\{\bar z^A\}\partial_{\bar w_\gamma},\\
&&g(\partial_{z_\alpha},\partial_{\bar z_{\beta}})=g(\partial_{w_\alpha},\partial_{\bar w_\beta})
+c_A^\beta\partial_{z_\alpha}\{z^A\}
+\bar c_A^\alpha\partial_{\bar z_\beta}\{\bar z^A\}+O(|z|^{n}),\\
&&g^{\vec z}(A,\beta)(P)=g^{\vec w}(A,\beta)(P)+\varepsilon(A)\cdot c_A^\beta\,.
\end{eqnarray*}
To ensure that $g^{\vec w}(A,\beta)(P)=0$ for all $A,\beta$, we solve the equations:
$$\varepsilon(A)c_A^\beta=g^{\vec z}(A,\beta)(P)\,.$$
Assertion~(2) now follows since the transformation is uniquely
defined if we suppose $dT(P)=\operatorname{id}$.
\end{proof}

We use Lemma~\ref{lem-2.1} to normalize the system of holomorphic coordinates $\vec z$
to arbitrarily high order henceforth; note that we also have:
$$g^{\vec z}(B;A)(P)=\bar g^{\vec z}(A;B)(P)=0\text{ for }|B|=1\,.$$
The structure group is now the unitary group $U(m)$ and the variables
$g^{\vec z}(A;B)$ are tensors; we shall suppress the role of the coordinate system
$\vec z$ whenever
no confusion is likely to result.
If we fix $|A|=n_1\ge2$ and $|B|=n_2\ge2$, then $g(\cdot;\cdot)$
is a symmetric cotensor of type $(n_1,n_2)$, i.e.
$$g(\cdot;\cdot)\in S^{n_1}(\Lambda^{1,0})
\otimes S^{n_2}(\Lambda^{0,1})\,.$$

The K\"ahler identity of Equation~(\ref{E1.b}) yields
$\RR(\partial_{z_a},\partial_{z_b})=\RR(\partial_{\bar z_a},\partial_{\bar z_b})=0$.
Let $A=(\alpha_1,\alpha_2)$ and $B=(\beta_1,\beta_2)$. We compute that:
\begin{eqnarray*}
&&R(\partial_{z_{\alpha_1}},\partial_{\bar z_{\beta_1}},
\partial_{\bar z_{\beta_2}},\partial_{z_{\alpha_2}})(P)\\
&=&\textstyle\frac12\{\partial_{z_{\alpha_1}}\partial_{\bar z_{\beta_2}}
g(\partial_{z_{\alpha_2}},\partial_{\bar z_{\beta_1}})
+\partial_{z_{\alpha_2}}\partial_{\bar z_{\beta_1}}
g(\partial_{z_{\alpha_1}},\partial_{\bar z_{\beta_2}})\}(P)\\
&=&g(A;B)(P)\,.
\end{eqnarray*}
A similar computation shows for $A=(\alpha_1,\alpha_2,\alpha_3)$ and
$B=(\beta_1,\beta_2)$ that:
\begin{eqnarray*}
&&\nabla R(\partial_{z_{\alpha_1}},\partial_{\bar z_{\beta_1}},
\partial_{\bar z_{\beta_2}},\partial_{z_{\alpha_2}};
\partial_{z_{\alpha_3}})(P)=g(A;B)(P)\,.
\end{eqnarray*}
The expression of the variables $g(A;B)(P)$ in terms of covariant derivatives of curvature
(and vice-versa) for larger values of $|A|$ and $|B|$ is more complicated.

\subsection{The K\"ahler potential}\label{S2.3}
Let
\begin{eqnarray*}
&&dz^I:=dz^{i_1}\wedge\dots\wedge dz^{i_p}\text{ for }I=\{1\le i_1<\dots<i_p\le m\},\\
&&d\bar z^J:=d\bar z^{j_1}\wedge \dots \wedge d\bar z^{j_q}\text{ for }J=\{1\le j_1<\dots<j_q\le m\}\,.
\end{eqnarray*}
We set $\Lambda^{p,q}M:=
\operatorname{Span}_{\mathbb{C}}\{dz^I\wedge d\bar z^J\}_{|I|=p,|J|=q}$ and decompose
$$\Lambda^nM\otimes_{\mathbb{R}}\mathbb{C}=\oplus_{p+q=n}\Lambda^{p,q}M\,.$$
Thus, for example, $\Lambda_+^2M\otimes_{\mathbb{R}}\mathbb{C}=\Lambda^{1,1}M$.
Decompose $d=\partial+\bar\partial$ where
$$\partial: C^\infty(\Lambda^{p,q}M)
\rightarrow C^\infty(\Lambda^{p+1,q}M)\text{ and }
\bar\partial:C^\infty(\Lambda^{p,q}M)\rightarrow C^\infty(\Lambda^{p,q+1}M)
$$
are defined by setting:
\begin{eqnarray*}
&&\partial(f_{I,J}dz^I\wedge d\bar z^J)=
\partial_{z_\alpha}(f_{I,J})dz^\alpha\wedge dz^I
\wedge d\bar z^J,\\
&&\bar\partial(f_{I,J}dz^I\wedge d\bar z^J):=
\partial_{\bar z_\alpha}(f_{I,J})d\bar z^\alpha
\wedge dz^I\wedge d\bar z^J\,.
\end{eqnarray*}
If $f\in C^\infty(M)$, define a real Hermitian symmetric bilinear form $h_f\in C^\infty(S_+^2)$
and a corresponding real anti-symmetric 2-form $\Omega_{h_f}\in C^\infty(\Lambda_+^2)$ by setting:
$$
\Omega_{h_f}=-\sqrt{-1}\partial\bar\partial f=-\sqrt{-1}\frac{\partial^2f}
{\partial_{z_\alpha}\partial_{\bar z_\beta}}dz^\alpha\wedge d\bar z^\beta\text{ and }
h_f=\frac{\partial^2f}{\partial_{z_\alpha}\partial_{\bar z_\beta}}dz^\alpha\circ d\bar z^\beta\,.
$$
We then have $d\Omega_{h_f}=0$ and, consequently, for small $\varepsilon$, $g+\varepsilon h_f$
is positive definite and thus a K\"ahler metric.

\subsection{Specifying the jets of the metric at $P$}\label{S2.4}
The variables $\{g(A;B)\}$ are a good choice of variables since, unlike the covariant derivatives of
the curvature tensor, there are no additional identities as the following result shows;
we are dealing with a pure polynomial algebra and we have avoided the Bianchi identities:

\begin{lemma}\label{lem-2.2}
Fix $n\ge2$. Let constants $c(A;B)\in\mathbb{C}$ be given
for $2\le|A|\le n$ and $2\le|B|\le n$ so that $c(A;B)=\bar c(B;A)$.
Let $P$ be a point of a K\"ahler manifold $(M,g_0,J)$. Use Lemma~\ref{lem-2.1}
to normalize the coordinate system $\vec z$ at $P$ so $g_0$
satisfies Equation~(\ref{eqn-2.b}). Then exists a K\"ahler metric $g$
on $(M,J)$ so that $g^{\vec z}$ also satisfies Equation~(\ref{eqn-2.b})
and so that $g^{\vec z}(A;B)(P)=c(A;B)$
for $2\le |A|\le n$ and $2\le |B|\le n$.
\end{lemma}

\begin{proof}
Let $\phi$ be a plateau function which is identically 1 for $|z|\le1$ and which vanishes identically
for $|z|\ge2$. Let $\phi_r(z):=\phi(z/r)$. Let $\varepsilon(\cdot)$ be the multiplicity
which was defined in
Equation~(\ref{eqn-2.c}). For $r$ small, we define:
$$f_r(z,\bar z)=\sum_{|A|=2}^n\sum_{|B|=2}^n\frac{c(A;B)-g^{\vec z}_0(A;B)(P)}
{\varepsilon(A)\varepsilon(B)}
\phi_r(z,\bar z)z^A\bar z^B\,.$$
The function $f_r$ is real and is supported arbitrarily close to $P$ for $r$ sufficiently small.
We follow the discussion of Section~\ref{S2.3} to define $h_f$. Let $g:=g_0+h_f$. Then
$$g_{\alpha\bar\beta}:=g_{0,\alpha\bar\beta}+
\sum_{|A|=2}^n\sum_{|B|=2}^n\frac{c(A;B)-g^{\vec z}_0(A;B)(P)}{\varepsilon(A)\varepsilon(B)}
\partial_{z_\alpha}\partial_{\bar z_\beta}
\left\{\phi_r(z,\bar z)z^A\bar z^B\right\}\,.$$
The perturbation has compact support near $P$; consequently, $g$
 extends smoothly to all of $M$. Furthermore,
since $\phi_r\equiv1$ near $P$,
$$g^{\vec z}(A;B)(P)=g_0^{\vec z}(A;B)(P)+c(A;B)-g^{\vec z}_0(A;B)(P)=c(A;B)\,.$$
Since $|A|\ge2$ and $|B|\ge2$, $g^{\vec z}$ satisfies Equation~(\ref{eqn-2.b}) at $P$.
Thus the only point remaining is to show that $g_{\alpha\bar\beta}$ is positive definite
if the parameter $r$ is chosen sufficiently small. Since $|A|\ge2$ and $|B|\ge2$,
there exists a constant $C$  so that if $r$ is small and if $|z|\le r$, we have:
$$\begin{array}{lll}
z^A\bar z^B\le Cr^4,&\partial_{z_\alpha}(z^A\bar z^B)\le Cr^3,&
\partial_{\bar z_\beta}(z^A\bar z^B)\le Cr^3,\\
\partial_{z_\alpha}\partial_{\bar z_\beta}(z^A\bar z^B)\le Cr^2,&
\phi_r\le C,&\partial_{z_\alpha}\phi_r\le Cr^{-1},\\
\partial_{\bar z_\beta}\phi_r\le Cr^{-1},&
\partial_{z_\alpha}\partial_{\bar z_\beta}\phi_r\le Cr^{-2}\,.\end{array}$$
After possibly increasing $C$, we may conclude that:
$$\partial_{z_\alpha}\partial_{\bar z_\beta}\{\phi_rz^A\bar z^B\}\le Cr^{2}\,.$$
Thus the perturbation of the original metric can be made arbitrary
small in the $C^0$ topology as $r\rightarrow0$ and hence
$g$ is positive definite if $r$ is sufficiently small.
\end{proof}

\section{The restriction map}\label{S3}
It is necessary to be somewhat more formal at this stage.
In Sectoin~\ref{S3}, we shall establish notation and make precise the notions
discussed previously in Definition~\ref{D1.3} and in Definition~\ref{D1.5}.

\begin{definition}
\rm
Let $\PPF_{ m}$ be the polynomial algebra in formal variables
$g(A;B)$ where $2\le|A|$ and $2\le|B|$.
Let
$\QQF_{ m}$ be the $\PPF_{ m}$ module
 of all $\QQ:=\PP^{\alpha\bar\beta}
 \partial_{z_\alpha}\circ\partial_{\bar z_\beta}$
which are $ S_+^2$ valued where $\PP^{\alpha\bar\beta}\in\PPF_{ m}$
for $1\le\alpha,\beta\le m$. If $\PP\in\PPF_m$ (resp. $\QQ\in\QQF_m$),
if $P$ is a point of K\"ahler manifold $\MMm$ of complex dimension $m$, and if $\vec z$ is
a system of local holomorphic coordinates on $M$ centered at $P$ satisfying the normalizations
of Lemma~\ref{lem-2.1}, then there is a natural evaluation $\PP(\MMm,\vec z)(P)$
(resp. $\QQ(\MMm,\vec z)(P)$). We use Lemma~\ref{lem-2.1} to see that we
can specify the variables $g(A;B)$ arbitrarily and therefore we may identify the abstract
element $\PP\in\PPF_m$ (resp. $\QQ\in\QQF_m)$ with the local formula it defines.
If $\PP(\MMm,\vec z)(P)=\PP(\MMm)(P)$ (resp.
$\QQ(\MMm,\vec z)(P)=\QQ(\MMm)(P)$) is independent
of the particular system of local holomorphic coordinates $\vec z$, then we say $\PP$ (resp. $\QQ$) is
{\it invariant}. Let $\PPFmU$ be the subalgebra and let $\QQFmU$ the $\PPFmU$ submodule of all such invariants.
The choice of $\vec z$ is unique up to the action
of $U(m)$. There is a natural dual action of $U(m)$ on $\PPF_m$ and $\QQF_m$;
$\PPFmU$ and $\QQFmU$ are simply
the fixed points of this action.
\end{definition}

A typical monomial $\AA$ of $\PP\in\PPF_m$ or of $\QQ\in\QQF_m$ takes the form:
$$
\AA=g(A_1^\AA;B_1^\AA)\dots g(A_\ell^\AA;B_\ell^\AA)
\partial_{z_{\alpha_\AA}}\circ\partial_{\bar z_{\beta_\AA}}\,.
$$
where we omit the $\partial_{z_{\alpha_\AA}}\circ\partial_{\bar z_{\beta_\AA}}$ variables
when dealing with an element of $\PPF_m$. Let
$c(\AA,\PP)$ (resp. $c(\AA,\QQ)$) be the coefficient of $\AA$ in $\PP$ (resp. $\QQ$);
we say that $\AA$  is a monomial
of $\PP$ (resp. $\QQ$) if $c(\AA,\PP)$ (resp. $c(\AA,\QQ)$) is non-zero.

\begin{definition}
\rm We introduce a grading on $\PPF_m$ and on $\QQF_{ m}$
by defining:
$$\ord(g(A;B)):=|A|+|B|-2\text{ and }
\ord(\AA)=\sum_i\{|A_i^{\AA}|+|B_i^{\AA}|-2\}\,.$$ The components of
$\RR$ have order $2$; the components of $\nabla\RR$ have order $3$,
and so forth. Let $T:=-\operatorname{id}\in U( m)$. Then
$T\AA=(-1)^{\ord(\AA)} \AA$. Thus if $\AA$ is a monomial of an invariant polynomial
$\PP$ or $\QQ$, then $\ord(\AA)$ is necessarily even. Decompose an invariant polynomial
$\PP=\PP_0+\PP_1+\dots$ where
$$
\PP_i:
=\sum_{\ord(\AA)=2i}c(\AA,\PP)
\AA\,.$$
Each $\PP_i$ is invariant separately since $U( m)$
preserves the order.
Let $\PPFmkU$ be the vector space of all elements of
$\PPFmU$
which are homogeneous of order $2k$ in the derivatives of the metric
and which are invariant under the action of the unitary group $U( m)$.
We define $\QQFmU$ and $\QQFmkU$ similarly. We may then decompose
$$\PPF_{ m}^{U}=\oplus_{ k}\PPFmkU\text{ and }
\QQFmU=\oplus_{ k}\QQFmkU\,.$$
\end{definition}

\begin{definition}
\rm Let $\deg_\gamma(A)$ be the number of times the
index $\gamma$ appears in a collection of indices $A$.
If
$$\AA_0=g(A_1^{\AA_0};B_1^{\AA_0})\dots g(A_\ell^{\AA_0};B_\ell^{\AA_0}),$$
let $\operatorname{len}(\AA_0):=\ell$ be the {\it length} of $\AA_0$.
Let $\deg_\gamma(\AA_0)$ (resp. $\deg_{\bar\gamma}(\AA_0)$) be the number of times
the holomorphic index $\gamma$ (resp. the anti-holomorphic index $\bar\gamma$)
appears in the monomial $\AA_0$:
\begin{eqnarray*}
&&\deg_\gamma(\AA_0)=\deg_\gamma(A_1^{\AA_0})+\dots+\deg_\gamma(A_\ell^{\AA_0}),\\
&&\deg_{\bar\gamma}(\AA_0)
=\deg_{\bar\gamma}(B_1^{\AA_0})+\dots+\deg_{\bar\gamma}(B_\ell^{\AA_0})\,.
\end{eqnarray*}
Similarly, if $\AA=\AA_0\partial_{z_{\alpha_\AA}}\circ\partial_{\bar z_{\beta_\AA}}$, set
$$\deg_\gamma(\AA):=\deg_\gamma(\AA_0)+\delta_{\gamma\alpha_\AA}\text{ and}
\deg_{\bar\gamma}(\AA):=\deg_{\bar\gamma}(\AA_0)+\delta_{\gamma\beta_\AA}\,.$$
\end{definition}

We wish to consider the space of universal scalar
valued curvature identities $\KKKFPmkU$
(resp. $ S_2^+$ valued curvature identities $\KKKFQmkU$)
which are homogeneous of order $2k$ in the derivatives
of the metric, which are defined on a manifold of complex dimension
$m\ge k$ (resp. $m\ge k+1$), and which
vanish when restricted to a manifold of complex dimension $k-1$ (resp. of complex dimension $k$).
We define these spaces algebraically as follows to give precision to the notation introduced previously
in Definition~\ref{D1.3} and in Definition~\ref{D1.5}.

\begin{definition}\label{defn-3.4}
\rm Define the {\it restriction map}
$$r_{ m,\nu}\{\AA\}:=\left\{\begin{array}{ll}
\AA&\text{ if }\deg_\alpha(\AA)=\deg_{\bar\alpha}(\AA)=0\text{ for all }\alpha> \nu\\
0&\text{ otherwise}
\end{array}\right\}\,.
$$
We note that $r_{ m,\nu}\{\AA\}$ is then a monomial in complex dimension $\nu$ so we may
extend $r_{m,\nu}$ to an algebra homomorphism and to a module homomorphism, respectively:
$$
r_{m,\nu}:\PPFmkU\rightarrow\PPF_{\nu, k}^{U}\text{ and }
r_{m, \nu}:\QQFmkU\rightarrow\QQF_{\nu, k}^{U}\,.
$$\end{definition}
There is an equivalent geometric formulation.
Let $\TT^{\ell}:=(\mathbb{T}^{\ell},g_{\mathbb{T}},J_{\mathbb{T}})$ be the flat
K\"ahler torus of complex dimension $\ell$ where
$\mathbb{T}^\ell:=\mathbb{R}^{2\ell}/\mathbb{Z}^{2\ell}$ is the rectangular torus of total
volume $1$, where $g_{\mathbb{T}}$ is the flat metric induced by the usual Euclidean metric, and
where $J_{\mathbb{T}}$ is the complex structure induced from the usual complex structure obtained
by identifying $\mathbb{R}^{2\ell}=\mathbb{C}^\ell$. Fix a base point $Q$ of $\mathcal{T}^\ell$.
The group of translations acts transitively on $\mathcal{T}^\ell$ so the particular base point chosen
is inessential.
The following Lemma gives an equivalent algebraic representation
of the spaces of universal curvature identities $\KKKFPmkU$ and $\KKKFQmkU$
which were discussed in Definition~\ref{D1.3} and
in Definition~\ref{D1.5}.

\begin{lemma}\label{L3.1}
Let $\nu<m$.
Let $P$ be a point of a K\"ahler manifold
$\NNv$ of complex dimension $\nu$.\begin{enumerate}
\item If $\PPP\in\PPFmkU$, then
$\PPP(\NNv\times\mathcal{T}^{m-\nu})(P,Q)
=(r_{m,\nu}\PPP)(\NNv)(P)$.
\item Let $i(P):=(P,Q)$ be the natural inclusion map of $N^\nu$ into $N^\nu\times T^{m-\nu}$.
If $\QQQ\in\QQFmkU$, then
$i^*\QQQ(\NNv\times\mathcal{T}^{m-\nu})(P,Q)
=(r_{m,\nu}\QQQ)(\NNv)(P)$.
\item $\KKKFPmkU=\ker(r_{m,k-1})\cap\PPFmkU$ and $\KKKFQmkU=\ker(r_{m,k})\cap\QQQFU$.
\end{enumerate}\end{lemma}

\noindent{\bf Note:}
 It is necessary to use the pull-back $i^*$ to regard the symmetric $2$-tensor
$P\rightarrow\QQQ(\NNv\times\mathcal{T}^{m-\nu})(P,Q)$
as a symmetric $2$-tensor on $\NNv$. But it
is not necessary to use pull-back to regard the function
$P\rightarrow\PPP(\NNv\times\mathcal{T}^{m-\nu})(P,Q)$
as a function on $\NNv$ so we shall omit the $i^*$ in that setting.

\begin{proof} Let $\MMm:=\NNv\times\mathcal{T}^{m-\nu}$. Any polynomial in the derivatives
of the metric which involves an index greater than $\nu$ vanishes since the metric is flat on
$\mathcal{T}^{m-\nu}$. Since we have restricted the symmetric $2$-tensors to
$\NNv$, a symmetric 2-tensor also vanishes
if it contains a holomorphic (or an anti-holomorphic) index greater than $\nu$. Assertion~(1)
and Assertion~(2) now follow.
Lemma~\ref{lem-2.2} permits us to identify an invariant polynomial (which is an algebraic
object) with the corresponding geometric formula it defines; Assertion~(3) now follows.
\end{proof}

We can now relate the restriction maps $r_{m,\nu}$ on $\xSSFm$
of Definition~\ref{D1.1} to the restriction maps $r_{m,\nu}$ on $\PPFmU$ and on $\QQFmU$
of Definition~\ref{defn-3.4}:
\begin{lemma}\label{L3.2}
\
\begin{enumerate}
\item Let $\XiPmk$ be as defined in Definition~\ref{D1.4}. \begin{enumerate}
\item If $m>\nu$, then $r_{m,\nu}\XiPmk=\Xi_{\mathfrak{P},\nu,k} r_{m,\nu}$ on $\xSSSFmk$.
\item If $m\ge k$, then $\XiPmk\xSSSFmk\subset\KKKFPmkU$.
\item If $m\ge k$ and if $0\ne\xSSmk\in\xSSSFmk$, then $r_{m,k}\XiPmk\xSSmk\ne0$.
\end{enumerate}
\item Let $\XiQmkT$ be as defined in Definition~\ref{D1.6}.
\begin{enumerate}
\item If $m>\nu$, then $r_{m,\nu}\XiQmkT=\Xi_{\mathfrak{Q},\nu,k} r_{m,\nu}$ on $\xSSSFmk$.
\item If $m\ge k+1$, then $\XiQmkT\xSSSFmk\subset\KKKFQmkU$.
\item If $m\ge k+1$ and if $0\ne\xSSmk\in\xSSSFmk$, then $r_{m,k+1}\XiQmkT\xSSmk\ne0$.
\end{enumerate}
\end{enumerate}
\end{lemma}

\begin{proof} Recall that $\XiPmk(\xSSmk)=\textstyle\frac1{ k!}g(\xSSmk(\RR),\Omega^{ k})$.
Assertion~(1a) is now immediate. Furthermore since $\Omega^k$ vanishes on a K\"ahler
manifold of complex dimension $k-1$, $\Xi_{\mathfrak{P},k-1,k}=0$.
By Assertion~(1a), $r_{m,k-1}\XiPmk=\Xi_{\mathfrak{P},k-1,k}r_{m,k-1}=0$. By Lemma~\ref{L3.1},
$\KKKFPmkU=\ker(r_{m,k-1})\cap\PPFmkU$. Assertion~(1b) now follows.
By Remark~\ref{R1.1},
$r_{m,k}$ is an isomorphism from $\xSSSFmk$ to $\xSSSFkk$. Thus to prove Assertion~(1c),
it suffices to show that $\XiPkk$ is injective from $\xSSSFkk$ to $\PPFkkU$.
We use Equation~(\ref{E1.a}) and Definition~\ref{D1.4} to see that:
$$
\XiPkk(\xSSkk)(\mathcal{R})d\nu_g=\textstyle\frac1{ k!}g(\xSSkk(\RR),\Omega^{ k})d\nu_g=
\xSSkk(\mathcal{R})\,.
$$
If $\xSSkk\ne0$, we may apply Lemma~\ref{L1.1} establish Assertion~(1c) by choosing $\vec\nu$ so that
$$\int_{\mathbb{CP}^{\vec\nu}}\xSSkk(\mathcal{R}_{\mathbb{CP}^{\vec\nu}})\ne0\,.$$

 Recall that
$\XiQmkT(\xSSmk)={\textstyle\frac1{(k+1)!}}g(\xSSmk(\RR)\wedge e^\alpha\wedge\bar e^\beta,
\Omega^{k+1})e_\alpha\circ\bar e_\beta$.
Assertion~(2a) is now immediate. Since $\Omega^{k+1}$ vanishes on a K\"ahler
manifold of complex dimension $k$, $\Xi_{\mathfrak{Q},k,k}=0$.
By Assertion~(2a),
$r_{m,k}\XiQmkT=\Xi_{\mathfrak{Q},k,k}r_{m,k}=0$. By Lemma~\ref{L3.1},
$\KKKFQmkU=\ker(r_{m,k})\cap\QQQFU$. Assertion~(2b) now follows.
By Remark~\ref{R1.1},
$r_{m,k}$ is an isomorphism from $\xSSSFmk$ to $\mathfrak{S}_{k+1,k}$. Thus to prove Assertion~(2c),
we may take $m=k+1$. Let $\MMK:=\NNk\times\mathcal{T}^1$ where $\mathcal{T}^1$
is the flat K\"ahler torus of complex dimension 1.
Let $w$ be the usual periodic complex parameter on $\mathbb{T}^1$.
\begin{eqnarray*}
&&\textstyle\frac1{(k+1)!}\Omega_{\MM}^{k+1}=\frac1{(k+1)!}(\Omega_{\MM}+
\Omega_{{{\mathcal{T}}}})^{k+1}
=\frac1{k!}\Omega_{\MM}^k\wedge\Omega_{{{\mathcal{T}}}},\\
&&\Xi_{\mathfrak{Q},k+1,k}(\mathcal{S}_{k+1,k})(\MMK)=\left\{
\Xi_{\mathfrak{P},k,k}(r_{k+1,k}\mathcal{S}_{k+1,k})(\NNk)\right\}\partial_w\circ\partial_{\bar w}\,.
\end{eqnarray*}
Because $r_{k+1,k}$ is an injective map from
$\mathfrak{S}_{k+1,k}$ to $\xSSSFkk$,
Assertion~(2c) follows from Assertion~(1c).\end{proof}

\begin{lemma}\label{L3.2x}
Let $\TQmk$ be as defined in Definition~\ref{D1.7}.
\begin{enumerate}
\item If $m>\nu$, then $r_{m,\nu}\TQmk=\Theta_{\mathfrak{Q},\nu,k} r_{m,\nu}$ on $\xSSSFmk$.
\item If $m\ge k+1$, then $\TQmk\xSSSFmk\subset\KKKFQmkU$.
\item If $m\ge k+1$ and if $0\ne\xSSmk\in\xSSSFmk$, then $r_{m,k+1}\TQmk\xSSmk\ne0$.
\end{enumerate}\end{lemma}

\begin{proof} It is necessary to expand the category in which we are working, if only briefly.
Let $\MMm=(M,g,J)$ be a Hermitian manifold of complex dimension $m$.
Let $\nabla^g$ be the associated Levi-Civita
connection. We average over the action of the complex structure $J$ to define
an auxiliary connection $\tilde\nabla^g:=(-J\nabla^gJ+\nabla^g)/2$
on the tangent bundle. It is immediate that $\tilde\nabla^gJ=J\tilde\nabla^g$ and thus $\tilde\nabla^g$
is a complex connection. The associated curvature $\mathcal{R}(\tilde\nabla^g)$ is then a complex
endomorphism and consequently $\xSSmk(\mathcal{R}(\tilde\nabla^g))\in\Lambda^{2k}(M)$ is well defined
and we may extend Definition~\ref{D1.4}, Definition~\ref{D1.6}, and Definition~\ref{D1.7} to
this setting. If $\MMm_\epsilon$ is a Hermitian variation, then  $\TQmk(\xSSmk)$ is characterized by the identity:
$$
\partial_{\varepsilon}\left.\left\{\int_M\XiPmk({{\xSSmk}})(\mathcal{R}_{\MM_\varepsilon})
d\nu_{\MM_\epsilon}\right\}\right|_{\varepsilon=0}
=\int_M\left\langle
\left\{\TQmk(\xSSmk)\right\}(\RR_\MM),h\right\rangle d\nu_g\,.
$$
Let $m>\nu$. We consider a product of the form
$\MMm_\epsilon=\NNv_\epsilon\times\mathcal{T}^{m-\nu}$
where the variation is trivial on the K\"ahler torus and where $\NNv_\epsilon$ is a
Hermitian variation. Since $\mathcal{T}^{m-\nu}$ has
unit volume, we can ignore the integral over the torus
and apply Lemma 3.1 and Lemma 3.2
to compute:
\begin{eqnarray*}
&&\partial_{\varepsilon}\left.\left\{\int_M\XiPmk(\xSSmk)(\mathcal{R}_{\MM_\varepsilon})
d\nu_{\MM_\epsilon}\right\}\right|_{\varepsilon=0}\\
&&\qquad=\int_M\left\langle
\left\{\TQmk(\xSSmk)\right\}(\RR_\MM),h\right\rangle d\nu_g\\
&&\qquad=\int_N\left\langle
\left\{r_{m,\nu}\TQmk(\xSSmk)\right\}(\RR_\NN),h\right\rangle d\nu_g.
\end{eqnarray*}
 We may also compute:
\begin{eqnarray*}
&&\partial_{\varepsilon}\left.\left\{\int_N\Xi_{\mathfrak{P},\nu,k}(r_{m,\nu}\xSSmk)
(\mathcal{R}_{\NN_\varepsilon})
d\nu_{\NN_\epsilon}\right\}\right|_{\varepsilon=0}\\
&&\qquad=\int_N\left\langle
\left\{\Theta_{\mathfrak{Q},\nu,k}(r_{m,\nu}\xSSmk)\right\}(\RR_\NN),h\right\rangle d\nu_g\,.
\end{eqnarray*}
This shows
$$0=\int_n\langle\left\{r_{m,\nu}\TQmk(\xSSmk)
-\Theta_{\mathfrak{Q},\nu,k}(r_{m,\nu}\xSSmk)\right\}(\RR_\NN),h\rangle d\nu_g\,.
$$
Since it is not necessary to restrict to K\"ahler variations, we can complete the proof of Assertion~(1)
by taking $h$ to be the dual of
$$\left\{r_{m,\nu}\TQmk(\xSSmk)
-\Theta_{\mathfrak{Q},\nu,k}(r_{m,\nu}\xSSmk)\right\}(\RR_\NN)$$
with respect to the metric $g$ to obtain
$$0=\int_N||\left\{r_{m,\nu}\TQmk(\xSSmk)
-\Theta_{\mathfrak{Q},\nu,k}(r_{m,\nu}\xSSmk)\right\}(\RR_\NN)||_g^2d\nu_g\,.$$
In complex dimension $k$, $\XiPkk(\xSSkk)[M]$
is a characteristic number and, consequently, since we constructed complex connections,
$\XiPkk(\xSSkk)[M]$ is
independent of the particular Hermitian metric chosen. This shows the
Euler Lagrange Equations are trivial and thus $\Theta_{\mathfrak{Q},k,k}=0$. Assertion~(2) now follows
from Assertion~(1).

We return to the K\"ahler setting and, by Assertion~(1), take $m=k+1$ in proving Assertion~(3).
Let $\MMK:=\NNk\times\mathcal{T}^1$ where $\mathcal{T}^1$
is the flat K\"ahler torus of complex dimension 1.
Let $w$ be the usual periodic complex parameter on $\mathbb{T}^1$. We take a variation of the form
$g_\varepsilon:=g_{\NN}+(1+\varepsilon)dw\circ d\bar w$. The curvature is unchanged
but we have $d\nu_\varepsilon=(1+\epsilon)d\nu_{\NN}d\nu_{\mathcal{T}}$. Consequently,
$$
\TQKk(\xSSKk)(\MMK)=\left\{\XiPkk(r_{k+1,k}\xSSKk)(\NNk)\right\}\partial_w\circ\partial_{\bar w}
$$
and Assertion~(3) follows from Assertion~(1c) of Lemma~\ref{L3.2}.
\end{proof}

\section{The action of the unitary group}\label{S4}

In this section, we use unitary invariance to study the spaces
$\PPFmU$ and $\QQFmU$. We then examine the spaces of universal curvature
identities $\KKKFPmkU$ and $\KKKFQmkU$ and obtain a fundamental estimate for
their dimensions.

\begin{lemma}\label{L3.3}
Let $\UU\in\PPFmU$ or let $\UU\in\QQFmU$. Let $\AA$ be a monomial of $\UU$. Express
$$\AA=g(A_1^{\AA};B_1^{\AA})\dots
g(A_\ell^{\AA};B_\ell^{\AA})\partial_{z_{\alpha_\AA}}\circ\partial_{\bar z_{\beta_\AA}}$$
where we omit the $\partial_{z_{\alpha_\AA}}\circ\partial_{\bar
z_{\beta_\AA}}$ variables if $\AA\in\PPFmU$.  Set $\operatorname{len}(\AA)=\ell$.
\begin{enumerate}
\item If $1\le\alpha\le m$, then $\deg_\alpha(\AA)=\deg_{\bar\alpha}(\AA)$.
\item
Assume that
$\deg_\alpha(\AA)>0$. Fix $\beta\ne\alpha$
and create a monomial $\tilde{\AA}$ by changing exactly one holomorphic index in
$\AA$ $\alpha\rightarrow\beta$.
Then there is a monomial $\AA_1$ of $\UU$ which is different from
$\AA$ and which also can create $\tilde{\AA}$ either by
changing exactly one holomorphic index $\alpha\rightarrow\beta$ or by
changing exactly one anti-holomorphic index $\bar\beta\rightarrow\bar\alpha$.
\item If $\UU\in\PPFmU$, then  there exists a monomial $\AA$ of $\UU$ so
$\deg_\alpha(\AA)=0$ for $\alpha>\operatorname{len}(\AA)$.
\item If $\UU\in\QQFmU$, then there exists a monomial $\AA$ of $\UU$ so
$\deg_\alpha(\AA)=0$ for $\alpha>\operatorname{len}(\AA)+1$.
\end{enumerate}
\end{lemma}

\begin{proof}
Fix $1\le\alpha\le m$ and consider the unitary transformation:
\begin{eqnarray*}
&&T_\alpha(\partial_{z_\gamma}):=\left\{\phantom{..}\begin{array}{rll}
e^{\sqrt{-1}\theta}\partial_{z_\gamma}&\text{if}&\gamma=\alpha\\
\partial_{z_\gamma}&\text{if}&\gamma\ne\alpha\end{array}\right\}\,,\\
&&T_\alpha(\partial_{\bar z_\gamma}):=\left\{\begin{array}{rll}
e^{-\sqrt{-1}\theta}\partial_{\bar z_\gamma}&\text{if}&\gamma=\alpha\\
\partial_{\bar z_\gamma}&\text{if}&\gamma\ne\alpha\end{array}\right\}\,.
\end{eqnarray*}
Then $T_\alpha\AA=e^{\sqrt{-1}\theta\{\deg_\alpha(\AA)-\deg_{\bar\alpha}(\AA)\}}\AA$,
so we have
$$
T_\alpha\UU=\UU
=\sum_{\AA}c(\AA,\UU)e^{\sqrt{-1}\theta\{\deg_\alpha(\AA)-\deg_{\bar\alpha}(\AA)\}}\AA\,.
$$
As $\theta$ was arbitrary, $c(\AA,\UU)\ne0$ implies
$\deg_\alpha(\AA)=
\deg_{\bar\alpha}(\AA)$. Assertion~(1) follows.

We now prove Assertion~(2). Fix indices $\alpha$ and $\beta$.
Set:
\begin{eqnarray*}
&&\nu:=\deg_\alpha(\AA)+\deg_\beta(\AA)=
\deg_{\bar\alpha}(\AA)+\deg_{\bar\beta}(\AA),\\
&&\tilde\UU:=\sum_{\BB :\deg_\alpha(\BB )+\deg_\beta(\BB )
=\nu}c(\BB ,\UU)\BB \,.
\end{eqnarray*}
Then $\tilde\UU$ is invariant under the action of $U(2)$ on the indices
$\{\alpha,\beta\}$ and we work with $\tilde\UU$ henceforth in the proof of Assertion~(2);
each monomial of $\tilde\UU$ is homogeneous of degree $\nu$
in $\{\alpha,\beta\}$ and also in $\{\bar\alpha,\bar\beta\}$.
Let $\tilde{\AA}$ be obtained from $\AA$
by changing a single holomorphic index $\alpha\rightarrow\beta$. Since
$$\deg_\alpha(\tilde{\AA})=\deg_\alpha(\AA)-1
=\deg_{\bar\alpha}(\AA)-1=\deg_{\bar\alpha}(\tilde{\AA})-1\,,$$
Assertion~(1) implies $\tilde{\AA}$ is not a monomial of $\tilde\UU$.
Let $u,v\in\mathbb{C}$ satisfy $|u|^2+|v|^2=1$. Consider
the unitary transformation
\begin{equation}\label{E4.b}
\begin{array}{l}
T\partial_{z_\sigma}=\left\{\begin{array}{rll}
\partial_{z_\sigma}&\text{ if }\sigma\ne\alpha,\beta\\
u\partial_{z_\alpha}+v\partial_{z_\beta}&\text{ if }\sigma=\alpha\\
-\bar v\partial_{z_\alpha}+\bar u\partial_{z_\beta}&\text{ if }\sigma=\beta
\end{array}\right\},\\
T\partial_{\bar z_\sigma}=\left\{\begin{array}{rll}
\partial_{\bar z_\sigma}&\text{ if }\sigma\ne\alpha,\beta\\
\bar u\partial_{\bar z_\alpha}+\bar v\partial_{\bar z_\beta}&\text{ if }\sigma=\alpha\\
-v\partial_{\bar z_\alpha}+u\partial_{\bar z_\beta}&\text{ if }\sigma=\beta
\end{array}\right\}\,.\end{array}
\end{equation}
We may expand
$$T\tilde\UU=f(u,v,\bar u,\bar v)\tilde{\AA}+\text{other terms}$$
where $f$ is homogeneous of degree $2\nu$ in $\{u,v,\bar u,\bar
v\}$; since $T\tilde\UU=\tilde\UU$ and since $\tilde{\AA}$ is not a
monomial of $\tilde{\UU}$, $f(u,v,\bar u,\bar v)=0$ for
$|u|^2+|v|^2=1$. Since $f$ is homogeneous, $f(u,v,\bar u,\bar v)$
vanishes for all $(u,v)$ and thus is the trivial polynomial. We have
$T\AA=n_{\AA,\tilde\AA} v u^{\nu-1}\bar u^\nu\tilde{\AA}+\dots$
where $n_{\AA,\tilde\AA}$ is a positive integer which reflects the
number of ways that $\AA$ can transform to $\tilde{\AA}$ by changing
a single holomorphic index $\alpha\rightarrow\beta$. There must
therefore be some monomial $\AA_1$ of $\UU$ which is
different from $\AA$ and which transforms to $\tilde{\AA}$ to create
a term involving $vu^{\nu-1}\bar u^\nu\tilde{\AA}+\dots$ and which
helps to cancel the corresponding term in $T\AA$. In view of
Equation~(\ref{E4.b}), this can only be by changing a holomorphic
index
 $\alpha\rightarrow\beta$ or an anti-holomorphic index
$\bar\beta\rightarrow\bar\alpha$. Assertion~(2) now follows.

We now prove Assertions~(3) and (4). We first introduce some
additional notation. Choose
$\nu=\nu(\AA)$ maximal among all possible rearrangements defining $\AA$
so
$$\deg_\alpha(A_i^{\AA})=0\text{ for }i<\alpha\text{ and }1\le i\le\nu\,.$$
If $\nu(\AA)=\ell$, go on to the next step.
If $\nu<\ell$, choose $\AA$ to be a monomial of  $\UU$ so that
$\nu(\AA)$ is maximal. Amongst all such possibilities choose $\AA$ so
that $\deg_{\nu+1}(A_{\nu+1}^\AA)$ is maximal.
Since $\nu(\AA)<\ell$, there is some index $\alpha>\nu+1$ so
$\deg_\alpha(A_{\nu+1}^{\AA})>0$. By making a coordinate permutation,
we may assume $\alpha=\nu+2$. Let $\AA=A_{\nu+1}^\AA A_0$.
Define $A_{\nu+1}^{\tilde{\AA}}$ by changing one holomorphic index
$\nu+2$ to $\nu+1$ in $A_{\nu+1}^{\AA}$ and let $\tilde\AA=A_{\nu+1}^{\tilde{\AA}}\AA_0$.
Apply Assertion~(2) to construct a monomial $\AA_1\ne\AA$ of $\UU$.
There are two possibilities:
\begin{enumerate}
\item $\AA_1$ transforms to $\tilde{\AA}$ by changing a holomorphic index
$\nu+2\rightarrow \nu+1$. Since
$\deg_\alpha(A_1^{\AA})=\dots=\deg_\alpha(A_{\nu}^{\AA})=0$ for $\alpha>\nu$,
$A_i^{\AA_1}=A_i^{\AA}$ for $i\le \nu$. Since $\AA_1\ne\AA$,
$A_{\nu+1}^{\AA_1}\ne A_{\nu+1}^{\AA}$. Consequently, $\nu(\AA_1)=\nu$ and
$\deg_{\nu+1}(A_{\nu+1}^{\AA_1})>\deg_{\nu+1}(A_{\nu+1}^{\AA})$. This contradicts the choice
of $\AA$ with $\nu(\AA)=\nu$ and $\deg_{\nu+1}(A_{\nu+1}^{\AA})$ maximal.
Thus this possibility is impossible.
\item $\AA_1$ transforms to $\tilde{\AA}$ by changing an anti-holomorphic index
$\bar \nu\rightarrow\overline{\nu+1}$. Then $A_i^{\AA_1}=A_i^{\tilde{\AA}}$ for all $i$.
Thus $\nu(\AA_1)=\nu$ and
$\deg_\nu(A_\nu^{\AA_1})>\deg_\nu(A_\nu^{\AA})$ which is impossible.
\end{enumerate}

The contradiction derived above shows we may choose $\AA$ so
$\deg_\alpha(A_i^{\AA})=0$ for $\alpha>\ell$ and $i\le\ell$. If $\UU\in\PPFmkU$, then
Assertion~(3) follows. Suppose $\UU\in\QQFmkU$.
If $\alpha_\AA\le \ell+1$, then we are done. If $\alpha_\AA>\ell+1$, we may interchange
the index $\alpha_\AA$ and the index $\ell+1$ to assume $\alpha_\AA=\ell+1$. This completes
the proof of Assertion~(4).
\end{proof}

The following technical Lemma is crucial to our study of the spaces of universal curvature identities
$\KKKFPmkU=\ker(r_{m,k-1})\cap\PPFmkU$ and $\KKKFQmkU=\ker(r_{m,k})\cap\QQQFU$.

\begin{lemma}\label{L4.2}
Let $\UU\in\KKKFPmkU$ or let $\UU\in\KKKFQmkU$.
Let
$$\AA=g(A_1^\AA;B_1^\AA)\dots g(A_\ell^\AA;B_{\ell}^\AA)\partial_{z_{\alpha_\AA}}
\circ\partial_{\bar z_{\beta_\AA}}$$
be a monomial of $\UU$; we omit the $\partial_{z_{\alpha_\AA}}
\circ\partial_{\bar z_{\beta_\AA}}$ variables if $\UU\in\KKKFPmkU$.
\begin{enumerate}
\item We have that $|A_i^\AA|=|B_i^\AA|=2$
and $\ell= k$.
\item There exists a monomial $\AA$ of $\UU$ satisfying:
\begin{enumerate}
\medbreak
\item For $1\le i\le k$, there exists an index $\alpha_i$ so that
$A_i^\AA=(\alpha_i,\alpha_i)$.
\item $\alpha_i=i$ for $1\le i\le k$.
\item If $\UU\in\KKKFQmkU$, then $\alpha_\AA= k+1$.
\item For $1\le i\le k$, there exists an index $\beta_i$ so that
$B_i^\AA=(\beta_i,\beta_i)$.
\item The indices $\{\beta_1,\dots,\beta_{ k}\}$ are a permutation of the indices
$\{1,\dots, k\}$.
\item If $\UU\in\KKKFQmkU$, then $\beta_\AA= k+1$.
\end{enumerate}\end{enumerate}
\end{lemma}

\begin{proof} The length $\operatorname{len}(\AA)=\ell$ of a monomial is unchanged by the action of $U( m)$.
Decompose
$$\UU=\UU_1+\UU_2+\dots\text{ where }
\UU_\ell:
=\sum_{\operatorname{len}(\AA)=\ell}c(\AA,\UU)\AA\,.
$$
Thus in proving Assertion~(1), we may suppose $\UU
=\UU_\ell$ for some $\ell$. Let $\AA$ be any monomial of $\UU$.
\begin{enumerate}
\item Suppose $\UU\in\KKKFPmkU=\ker(r_{m,k-1})\cap\PPFmkU$. By Lemma~\ref{L3.3}~(3),
we can choose a monomial
$\AA$ of $\UU$ so that no index other than $\{1,\dots,\ell\}$ appears in $\AA$.
As $r_{m,k-1}(\UU)=0$, there exists an index $\alpha\ge k$ so that
$\deg_\alpha(\AA)>0$. Consequently, $\ell\ge k$.
\item Suppose $\UU\in\KKKFQmkU=\ker(r_{m,k})\cap\QQQFU$. By Lemma~\ref{L3.3}~(4),
we can choose a monomial
$\AA$ of $\UU$ so that no index other than $\{1,\dots,\ell+1\}$ appears in $\AA$.
Since $r_{m,k,\QQF}(\UU)=0,$ there exists an index $\alpha\ge k+1$ so that
$\deg_\alpha(\AA)>0$.  This once again implies $\ell\ge k$.
\end{enumerate}
Since $|A_i^\AA|\ge2$ and $|B_i^\AA|\ge2$, we may estimate:
$$2 k=\ord(\AA)=\sum_{i=1}^{\ell}\left\{|A_i^\AA|+|B_i^\AA|-2\right\}\ge
2\ell\ge2 k\,.$$
Consequently, all these inequalities must have been equalities.
Thus shows that $|A_i^\AA|=|B_i^\AA|=2$ and therefore that $\UU$ only involves
the $2$-jets of the metric; the covariant derivatives of the curvature tensor
play no role. It also shows that $\ell= k$ so
$\operatorname{len}(\AA)= k$. Assertion~(1) now follows.

We shall assume $\UU=\QQ\in\KKKFQmkU=\ker(r_{m,k})\cap\QQQFU$
as the case in which $\UU\in\KKKFPmkU=\ker(r_{m,k-1})\cap\PPFmkU$ is similar.
We define
$$\QQ_{ k+1, k}=\sum_{\deg_\alpha(\AA)=0\text{ for }\alpha> k+1}
c(\UU,\AA)\AA\,.$$ This is invariant under the action of $U( k+1)$
and the argument given above shows $\QQ_{ k+1, k}\ne0$. Furthermore,
every index $\{1,\dots, k+1\}$ appears in every monomial of $\QQ_{
k+1, k}$ and thus $\QQ_{ k+1, k}\in \KKF_{\QQF, k+1, k}^{U}$. Finally,
every monomial of $\QQ_{k+1, k}$ is a monomial of $\UU$.
This shows that we may assume that the complex dimension is
$m=k+1$ in the proof of Assertion~(2); this is the crucial case. Thus every
monomial $\AA$ of $\QQ_{ k+1, k}$ contains as holomorphic indices
exactly the indices $\{1,\dots, k+1\}$ and also contains exactly these
indices as anti-holmorphic indices.

 We say that a holomorphic index $\alpha$ {\it touches itself in $\AA$} if we have
$A_i^\AA=(\alpha,\alpha)$ for some $i$. Choose a monomial $\AA$ of $\QQ_{ k+1, k}$ so
the number of holomorphic indices which touch themselves in $\AA$ is maximal.
By making a coordinate permutation,
we may assume without loss of generality the indices which touch themselves
holomorphically in $\AA$ are the indices $\{1,\dots,\nu\}$. Consequently
$A_i^\AA=(i,i)$ for $i\le\nu$. Suppose $\nu< k$.
Both the indices $\nu+1$ and $\nu+2$ appear holomorphically in $\AA$ since every
index $\{1,\dots, k+1\}$ appears in $\AA$. Since only one index can appear in
$\partial_{z_{\alpha_\AA}}$, we may assume that $A_{\nu+1}^\AA=(\nu+1,\sigma)$.
Furthermore, by the maximality of $\nu$, we have $\nu+1\ne\sigma$. Express
$$\AA=g(1,1;\star,\star)\dots g(\nu,\nu;\star,\star)g(\nu+1,\sigma;\star,\star)\AA_0$$
where ``$\star$" indicates indices not of interest and where $\AA_0$ is a suitably
chosen monomial.
We apply Lemma~\ref{L3.3}~(2)
to construct $\tilde\AA$ by changing a single holomorphic index $\sigma\rightarrow\nu+1$:
$$\tilde\AA=g(1,1;\star,\star)\dots g(\nu,\nu;\star,\star)g(\nu+1,\nu+1;\star,\star)\AA_0\,.$$
We apply Lemma~\ref{L3.3}~(2) to choose a monomial $\AA_1\ne\AA$ of
$\QQ_{ k+1, k}$. There are two possibilities:
\begin{enumerate}
\item If $\AA_1$ transforms to $\tilde\AA$ by changing an anti-holomorphic
index $\overline{\nu+1}$ to $\bar\sigma$,
then the holomorphic indices are unchanged and we have found a monomial $\AA_1$
of $\QQ_{ k+1, k}$ where one more index touches itself holomorphically.
This contradicts the choice of $\AA$ such that the number of indices touching themselves
holomorphically is maximal.
\item If $\AA_1$ transforms to $\tilde\AA$ by changing a holomorphic index $\sigma$ to $\nu+1$, then
we can not have changed $A_i^\AA$ for $i\le\nu$ since the index $\nu+1$ does not appear here.
Furthermore, since $A_{\nu+1}^{\tilde A}=(\nu+1,\nu+1)$, and since $\AA_1\ne\AA$,
that variable was not changed.
Thus
$$\AA_1=g(1,1;\star,\star)\dots g(\nu,\nu;\star,\star)g(\nu+1,\nu+1;\star,\star)\tilde\AA_0$$
and again, one more index touches itself holomorphically.
This contradicts the choice of $\AA$ such that the number of indices touching themselves
holomorphically is maximal.
\end{enumerate}

We have shown $\nu= k$. This establishes Assertion~(2a). Since every index must in fact appear in
$\AA$, no index can touch itself holomorphically in $\AA$ in two
different variables. Thus after permuting the indices appropriately,
we have that
$$\AA=g(1,1;\star,\star)\dots g( k, k;\star,\star)\partial_{z_{ k+1}}
\circ\partial_{\bar z_\star}\,.$$ This establishes
Assertion~(2b) and Assertion~(2c).

We will use the same argument to establish the remaining assertions; the
analysis is slightly more tricky since we do not want to destroy the
normalizations of Assertions~(2a) and (2b). Let $\AA$ be a
monomial of $\QQ_{ k+1, k}$ which satisfies the normalizations of
Assertions~(2a) and (2b). Let $\sigma\le k$. Then $\sigma$
appears twice holomorphically in $\AA$ and hence by
Lemma~\ref{L3.3}~(1) also appears anti-holomorpically in $\AA$
twice. The index $\sigma= k+1$ appears once holomorphically in $\AA$
and once anti-holomorphically in $\AA$. Choose $\AA$ so the number
$\nu$ of indices which touch themselves anti-holomorphically in
$\AA$ is maximal. If $\nu= k$, then we are done. So we assume
$\nu<k$ and argue for a contradiction. By permuting the indices, we may
assume the indices $\overline 1,\dots,\overline\nu$ touch themselves anti-holomorphically
in $\AA$ and that the index $\overline{\nu+1}$ does not touch itself
anti-holomorphically in $\AA$. Since $\overline{\nu+1}$ appears twice
anti-holomorphically, it must touch some other index $\bar x$
anti-holomorphically. Express:
$$\AA=g(\star,\star;\overline{\nu+1},\bar x)\AA_0\,.$$
Change the anti-holomorphic index $\bar x$ to an anti-holomorphic index
$\overline{\nu+1}$  to form:
$$\tilde\AA=g(\star,\star;\overline{\nu+1},\overline{\nu+1})\AA_0\,.$$
We use Lemma~\ref{L3.3}~(2)
to construct a monomial $\AA_1$ of $\QQ_{ k+1, k}$ different from
$\AA$. If $\AA_1$ transforms to $\tilde\AA$ by changing an
anti-holomorphic index $\bar x$ to the anti-holomorphic index $\overline{\nu+1}$, then
the fact that $i$ touches itself anti-holomorphically for $i\le\nu$ is not
spoiled and since $\AA\ne\AA_1$, $\overline{\nu+1}$ touches itself anti-holomorphically
in $\AA_1$. Since only the anti-holomorphic indices
are changed, the normalizations of Assertions~(2a) and (2b)
are not affected. Thus one more index would touch itself
anti-holomorphically in $\AA_1$ than is the case in $\AA$ and this
would contradict the maximality of $\nu$. Thus $\AA_1$ transforms to
$\AA$ by changing a holomorphic index $\nu+1$ to $x$. This destroys
the normalizations of Assertion~(2a). There are several
possibilities which we examine seriatim; we shall list the generic
case but if the variables collapse, this plays no role. In what
follows, we permit $x=y$.

\medbreak\noindent{\bf Case I:} The index $x$ appears once in $\AA$.
 Let $\star$ indicate
a term not of interest.
Let $\varepsilon$ be either a $\partial_{z_\alpha}\circ\partial_{\bar z_\beta}$ variable
or a $g(-,-;-,-)$ variable to have a uniform notation and to avoid
multiplying the cases unduly; we shall not fuss about the number of indices in $\varepsilon$
and thus the second $\star$ could be the empty symbol if $\epsilon(\star;\bar\beta,\star)$
indicates
the $\partial_{z_\alpha}\circ\partial_{\bar z_\beta}$ variable whereas the first $\star$
could indicate two indices if $\epsilon(\star;\bar\beta,\star)$ denotes a
$g(\star,\star;\bar\beta,\star)$ variable.
Let $\AA_0$ be an auxiliary monomial. We may express
\medbreak\qquad
$\AA=g(\nu+1,\nu+1;\star,\star)g(\star,\star;\overline{\nu+1},\bar x)
\varepsilon(\star;\overline{\nu+1},\star)\varepsilon(x,\star;\star)\AA_0$, where
\medbreak\qquad\qquad$\deg_{\nu+1} (\AA)=2,\quad\deg_{\overline{\nu+1}}(\AA)=2,\quad\deg_x(\AA)=1,\quad
\deg_{\bar x}(\AA)=1$
\medbreak\noindent
We change an anti-holomorphic index $\bar x$ to an anti-holomorphic index
$\overline{\nu+1}$ to construct:
\medbreak\qquad
$\tilde\AA=g(\nu+1,\nu+1;\star,\star)g(\star,\star;\overline{\nu+1},\overline{\nu+1})
\varepsilon(\star;\overline{\nu+1},\star)\varepsilon(x,\star;\star)\AA_0$, where
\medbreak\qquad\qquad
$\deg_{\nu+1} (\tilde\AA)=2,\quad\deg_{\overline{\nu+1}}(\tilde\AA)=3,\quad\deg_x(\tilde\AA)=1,\quad
\deg_{\bar x}(\tilde\AA)=0$.
\medbreak\noindent Since $\AA_1$ transforms to $\tilde\AA$
by changing a holomorphic index $\nu+1$ to a holomorphic
index $x$, $\deg_{\bar x}(\AA_1)=0$ which is impossible since every index from $1$
to $ k+1$ appears in every
monomial of $\QQ_{ k+1, k}$.

\medbreak\noindent{\bf Case II:}  The index $x$ appears twice in $\AA$ and does not
appear in $\partial_{\bar z_\beta}$.
Then
\medbreak\qquad
$\AA=g(\nu+1,\nu+1;\star,\star)g(x,x;\star,\star)g(\star,\star;\overline{\nu+1},\bar x)
g(\star,\star;\bar x,\bar z)$
\smallbreak\qquad\qquad\qquad
$\varepsilon(\star;\overline{\nu+1},\star)\varepsilon(\star;\bar z,\star)\AA_0$, where
\smallbreak\qquad\qquad
$\deg_{\nu+1} (\AA)=2,\quad\deg_{\overline{\nu+1}}(\AA)=2,\quad\deg_x(\AA)=2,\quad
\deg_{\bar x}(\AA)=2$,
\medbreak\qquad
$\tilde\AA=g(\nu+1,\nu+1;\star,\star)g(x,x;\star,\star)g(\star,\star;\overline{\nu+1},\overline{\nu+1})
g(\star,\star;\bar x,\bar z)$
\smallbreak\qquad\qquad\qquad$\varepsilon(\star;\overline{\nu+1},\star)\varepsilon(\star;\bar z,\star)\AA_0$, where
\smallbreak\qquad\qquad$\deg_{\nu+1} (\tilde\AA)=2,\quad\deg_{\overline{\nu+1}}(\tilde\AA)=3,
\quad\deg_x(\tilde\AA)=2,\quad
\deg_{\bar x}(\tilde\AA)=1$, and
\medbreak\qquad
$\AA_1=g(\nu+1,\nu+1;\star,\star)g(\nu+1,x;\star,\star)g(\star,\star;\overline{\nu+1},\overline{\nu+1})
g(\star,\star;\bar x,\bar z)$
\smallbreak\qquad\qquad\qquad
$\varepsilon(\star;\overline{\nu+1},\star)\varepsilon(\star;\bar z,\star)\AA_0$, where
\smallbreak
$\qquad\qquad\deg_{\nu+1} (\AA_1)=3,\quad\deg_{\overline{\nu+1}}(\AA_1)=3,\quad\deg_x(\AA_1)=1,\quad
\deg_{\bar x}(\AA_1)=1$.
\medbreak\noindent
We permit $z=\nu+1$.
We change an anti-holomorphic index $\bar x$ to $\bar z$ to create:
\medbreak\qquad
$\tilde\AA_1=g(\nu+1,\nu+1;\star,\star)g(\nu+1,x;\star,\star)g(\star,\star;\overline{\nu+1},\overline{\nu+1})
g(\star,\star;\bar z,\bar z)$
\smallbreak\qquad\qquad\qquad
$\varepsilon(\star;\overline{\nu+1},\star)\varepsilon(\star;\bar z,\star)\AA_0$, where
\smallbreak\qquad\qquad
$\deg_{\nu+1} (\tilde\AA_1)=3,\quad\deg_{\overline{\nu+1}}(\tilde\AA_1)=3,
\quad\deg_x(\tilde\AA_1)=1,\quad
\deg_{\bar x}(\tilde\AA_1)=0$.
\medbreak\noindent Again, we construct $\AA_2$.
 If we transform $\AA_2$ to $\tilde\AA_1$ by changing a
holomorphic index $z$ to a holomorphic index $x$, then
\medbreak\qquad\qquad
$\deg_{\nu+1} (\AA_2)=3,\quad\deg_{\overline{\nu+1}}(\AA_2)=3,\quad\deg_x(\AA_2)=0,\quad
\deg_{\bar x}(\AA_2)=0$.
\medbreak\noindent This contradicts the fact that $\deg_x(\AA_2)>0$.
Consequently $\AA_2$ transforms
to $\tilde\AA_1$ by changing an anti-holomorphic index $\bar x$ to an anti-holmorphic index
$\bar z$.
Since $\AA_2\ne\AA_1$,
\medbreak\qquad
$\AA_2=g(\nu+1,\nu+1;\star,\star)g(\nu+1,x;\star,\star)g(\star,\star;\overline{\nu+1},\overline{\nu+1})
g(\star,\star;\bar z,\bar z)$
\smallbreak\qquad\qquad\qquad
$\varepsilon(\star;\overline{\nu+1},\star)\varepsilon(\star;\bar x,\star)\AA_0$, where
\smallbreak\qquad\qquad$\deg_{\nu+1} (\AA_2)=3,\quad\deg_{\overline{\nu+1}}(\AA_2)=3,\quad\deg_x(\AA_2)=1,\quad
\deg_{\bar x}(\AA_2)=1$.
\medbreak\noindent
We have simply interchanged the anti-holomorphic indices $\bar x$ and $\bar z$
to construct $\AA_2$ from $\AA_1$.
We construct $\tilde\AA_2$ by changing a holomorphic index $\nu+1$ to $x$ to create:
\medbreak\qquad
$\tilde\AA_2=g(\nu+1,\nu+1;\star,\star)g(x,x;\star,\star)g(\star,\star;\overline{\nu+1},\overline{\nu+1})
g(\star,\star;\bar z,\bar z)$
\smallbreak\qquad\qquad\qquad
$\varepsilon(\star;\overline{\nu+1},\star)\varepsilon(\star;\bar x,\star)\AA_0$, where
\smallbreak\qquad\qquad$\deg_{\nu+1} (\tilde\AA_2)=2,\quad\deg_{\overline{\nu+1}}(\tilde\AA_2)=3,
\quad\deg_x(\tilde\AA_2)=2,\quad\deg_{\bar x}(\tilde\AA_2)=1$.
\medbreak\qquad
We consider $\AA_3$. Since $\AA_3\ne\AA_2$, $\AA_3$ does not transform to $\tilde\AA_2$
by changing a holomorphic index $\nu+1$ to $x$. Instead, $\AA_3$ transforms to $\tilde\AA_2$
by transforming an anti-holomorphic index $\bar x$ to an anti-holomorphic index
$\overline{\nu+1}$.
There are two possibilities
\medbreak\qquad
$\AA_3=g(\nu+1,\nu+1;\star,\star)g(x,x;\star,\star)g(\star,\star;\overline{\nu+1},\overline{\nu+1})
g(\star,\star;\bar z,\bar z)$
\smallbreak\qquad\qquad\qquad
$\varepsilon(\star;\bar x,\star)\varepsilon(\star;\bar x,\star)\AA_0$, or
\medbreak\qquad
$\AA_3=g(\nu+1,\nu+1;\star,\star)g(x,x;\star,\star)g(\star,\star;\overline{\nu+1},\bar x)
g(\star,\star;\bar z,\bar z)$
\smallbreak\qquad\qquad\qquad
$\varepsilon(\star;\overline{\nu+1},\star)
\varepsilon(\star;\bar x,\star)\AA_0$.
\medbreak\noindent Both these possibilities satisfy the normalization of
Assertion~(2a). And there is either one more
anti-holomorphic or two more anti-holomorphic indices which touch
themselves. This is impossible by the maximality of $\AA$.

\medbreak\noindent{\bf Case III:} The index $x$ appears twice in $\AA$
and appears in $\partial_{\bar z_\beta}$. Then $\nu+1 $ does not appear in
$\partial_{\bar z_\beta}$ and hence some other variable $g_{\star,\star;\nu+1,y}$ appears
in $\AA$. If $\deg_y(\AA)=1$, then Case I pertains. If $\deg_y(\AA)=2$, then Case II pertains.
This final contradiction establishes the Lemma.
\end{proof}

\subsection{The crucial estimate} Let $\rho( k)$ be the number of partitions of $ k$
as described in Definition~\ref{D1.1}.

\begin{lemma}\label{L4.3}
If $ m> k$, then $\dim\{\KKKFQmkU\}\le\rho( k)$ and $\dim\{\KKKFPmkU\}\le\rho( k)$.
\end{lemma}

 \begin{proof}
 Let $0\ne\QQQ\in\KKKFQmkU$. Apply Lemma~\ref{L4.2} to find a monomial $\AA$ of
$\QQQ$ so that
$$\AA_\sigma=g(1,1;\bar\sigma(1),\bar\sigma(1))g(2,2;\bar\sigma(2),\bar\sigma(2))\dots
g( k, k;\bar\sigma( k),\bar\sigma( k))\partial_{z_{ k+1}}\circ\partial_{\bar z_{ k+1}}$$
where $\sigma\in\operatorname{Perm}( k)$ is a suitably chosen permutation.
Thus $\QQQ\ne0$ implies $c(\AA_\sigma,\QQQ)\ne0$
for some $\sigma$. Only the conjugacy class of $\sigma$ in
$\operatorname{Perm}( k)$ is important
and, writing the permutation $\sigma$ in terms of cycles, we see that there are
$\rho( k)$ such conjugacy classes; ordering the lengths of these cycles in
decreasing order determines a partition $\pi$.
Thus there are $\rho( k)$ monomials $A_\pi$ so that $\QQQ\ne0$
implies $c(\AA_\pi)\ne0$; the inequality $\dim\{\KKKFQmkU\}\le\rho( k)$ now follows.
The proof of the inequality $\dim\{\KKKFPmkU\}\le\rho( k)$ is analogous and is therefore omitted.
\end{proof}

\section{The proofs of Theorem~\ref{T1.1}, Theorem~\ref{T1.2}, and
Theorem~\ref{T1.3}}\label{S5}

\subsection{The proof of Theorem~\ref{T1.1} and of Theorem~\ref{T1.2}}
Let $m\ge k$. By Lemma~\ref{L3.2}, $\XiPmk$ is a 1-1 map from $\xSSSFmk$ to
$\KKKFPmkU$. By Equation~(\ref{E1.d}), we have that $ \dim\{\xSSSFmk\}=\rho( k)$
By Lemma~\ref{L4.3}, $\dim\{\KKKFPmkU\}\le\rho(k)$.
Consequently
$$\dim\{\KKKFPmkU\}=\dim\{\xSSSFmk\}=\rho(k)$$ and
$\XiPmk$ is
an isomorphism. This proves Theorem~\ref{T1.1}. The same line of argument
shows that $\XiQmkT$ is an isomorphism from $\xSSSFmk$ to
$\KKKFQmkU$; this establishes Theorem~\ref{T1.2}.

\subsection{The proof of Theorem~\ref{T1.3}}
We must show $\TQmk=\XiQmkT$. Suppose to the contrary that
 $\TQmk\xSSmk\ne\XiQmkT\xSSmk$ for some $\xSSmk\in\xSSSFmk$.
We apply Lemma~\ref{L3.2} and Lemma~\ref{L3.2x} to see
$$0\ne r_{m,k+1}\{\TQmk-\XiQmkT\}\xSSmk=\{\TQmk-\XiQmkT\}(r_{m,k+1}\xSSmk)\,.$$
Thus we may suppose without loss of generality that $m=k+1$. We apply the argument used to establish
Lemma~\ref{L3.2}~(3). Let $\MMK_\epsilon:=\NNk\times{{\mathcal{T}_\epsilon^1}}$ where
the metric on {$\mathcal{T}_\epsilon^1$} is $(1+\epsilon)dw\circ d\bar w$.
Since the metric on $\NNk$
is unchanged and only the volume element on $M$ is changing,
\begin{equation}\label{E5.a}
\begin{array}{l}
\frac1{k!}g_\epsilon(\xSSKk(\mathcal{R}_{\MM_\epsilon}),\Omega_\epsilon^k)=
\frac1{k!}g(r_{k+1,k}\xSSKk(\mathcal{R}_\NN),\Omega_\NN^k),\\
\partial_\epsilon\left\{g_\epsilon(\xSSKk(\mathcal{R}_{\MM_\epsilon}),\Omega_\epsilon^k)\right\}=0,\gronk\\
\partial_\epsilon\left\{d\nu_{\MM_\epsilon}\right\}=d\nu_\MM=d\nu_\NN d\nu_{\mathcal{T}}
\,.\gronk
\end{array}\end{equation}
Since $\mathcal{T}^1$ has volume $1$, we may use Equation~(\ref{E5.a}) to compute:
\begin{equation}\label{E5.b}
\begin{array}{l}
\displaystyle
\partial_\epsilon\left.\left\{\frac1{k!}\int_Mg_\epsilon(\xSSKk(\mathcal{R}_{\MM_\epsilon}),\Omega_\epsilon^k)
d\nu_{\MM_\epsilon}\right\}\right|_{\epsilon=0}\\
\displaystyle\quad=
\frac1{k!}\int_Ng(r_{k+1,k}\xSSKk(\mathcal{R}_{\NN}),\Omega^k)d\nu_\NN\,.\vphantom{\vrule height 20pt}
\end{array}\end{equation}
Since $N$ has complex dimension $k$, we have
\begin{equation}\label{E5.c}
\frac1{k!}\int_Ng(r_{k+1,k}\xSSKk(\mathcal{R}_{\NN}),\Omega^k)d\nu_{\NN}
=\int_Nr_{k+1,k}\xSSKk(\mathcal{R}_\NN)\,.
\end{equation}

By Lemma~\ref{L3.2x}, $\TQKk\xSSKk\in\KKKFQKkU$. By Theorem~\ref{T1.2},
$\XiQKkT$ is an isomorphism from $\xSSSFKk$ to $\KKKFQKkU$. Thus we may find
$\tilde{\xSS}_{k+1,k} \in \xSSSFKk$ so that we have  $\XiQKkT{\tilde\xSS}_{k+1,k}=\TQKk\xSSKk$. Consequently:
\begin{equation}\label{E5.d}
\begin{array}{l}
\displaystyle\partial_\epsilon\left.\left\{\frac1{k!}\int_Mg(\xSSKk(\mathcal{R}_{\MM_\epsilon}),\Omega_\epsilon^k)
d\nu_{\MM_\epsilon}\right\}\right|_{\epsilon=0}\\
\qquad=\displaystyle\int_M\langle
\TQKk\xSSKk(\RR_\MM),h\rangle d\nu_g\vphantom{\vrule height 20pt}\\
\qquad=\displaystyle\int_M\langle
\XiQmkT{\tilde\xSS}_{k+1,k}(\RR_\MM),h\rangle d\nu_g\,.\vphantom{\vrule height 20pt}
\end{array}\end{equation}
We use the definition and the argument  used to establish Equation~(\ref{E5.c}) to compute:
\begin{eqnarray}
&&\int_M\langle\XiQKkT{\tilde\xSS}_{k+1,k}(\RR_\MM),h\rangle d\nu_g\nonumber\\
&=&\frac1{(k+1)!}\int_Mg(\tilde{\xSS}_{k+1,k}(\RR_\MM)\wedge e^\alpha\wedge\bar e^\beta,
\Omega_\MM^{k+1})\langle e_\alpha\circ\bar e_\beta,h\rangle {{d\nu_g}}\nonumber\\
&=&\frac1{k!}\int_Mg(r_{k+1,k}\tilde S_{k+1,k}(\mathcal{R}_\NN),\Omega_{\NN}^k)
d\nu_\NN d\nu_\TT\nonumber\\
&=&\int_Nr_{k+1,k}\tilde S_{k+1,k}(\mathcal{R}_\NN)\,.\label{E5.e}
\end{eqnarray}
We use Equation~(\ref{E5.b}),  Equation~(\ref{E5.c}), Equation~(\ref{E5.d}), and Equation~(\ref{E5.e})
to see
$$\int_Nr_{k+1,k}\{\xSSKk-\tilde{\xSS}_{k+1,k}\}{{(\mathcal{R}_\NN)}}=0\,.$$
Since $\NNk$ was an arbitrary K\"ahler manifold of complex dimension $k$, we may apply
Lemma~\ref{L1.1} to see $r_{k+1,k}\{\xSSKk-\tilde{\xSS}_{k+1,k}\}=0$. By Remark~\ref{R1.1},
$\xSSKk=\tilde{\xSS}_{k+1,k}$ and consequently $\XiQKkT\xSSKk=\TQKk\xSSKk$.
This completes the proof of Theorem~\ref{T1.3}.\hfill\qed

\subsection*{Acknowledgements} This work was supported by
the National Research Foundation of Korea (NRF) grant funded by the
Korea government (MEST)(2011-0012987),
by project MTM2009-07756 (Spain), and by project 174012 (Serbia).
% NRF-2011-0012987

 \end{document}